\documentclass[a4paper,12pt,reqno]{amsart}

\usepackage[T1]{fontenc}
\usepackage[utf8]{inputenc}
\usepackage{lmodern}
\usepackage[english]{babel}

\usepackage[%
	left=3.5cm,       
	right=3.5cm,      
	top=3.5cm,        
	bottom=3.5cm,     
	heightrounded,    
	bindingoffset=0mm 
]{geometry}

\usepackage{functan}
\usepackage{amssymb}
\usepackage{mathtools}
\usepackage{mathrsfs}

\usepackage[hyphens]{url}
\usepackage[
colorlinks=true,
urlcolor=teal,
citecolor=magenta,
linkcolor=teal,
linktocpage,
pdfpagelabels,
bookmarksnumbered,
bookmarksopen,
breaklinks=true]
{hyperref}

\usepackage[
capitalize, 
nameinlink,
noabbrev]
{cleveref}

\usepackage{bookmark}
\usepackage[initials,msc-links]{amsrefs}

\usepackage[dvipsnames]{xcolor}
\usepackage{graphicx}
\usepackage{enumitem}

\numberwithin{equation}{section}

\theoremstyle{plain}
\newtheorem{theorem}{Theorem}[section]
\newtheorem{lemma}[theorem]{Lemma}
\newtheorem{proposition}[theorem]{Proposition}
\newtheorem{corollary}[theorem]{Corollary}

\theoremstyle{definition}
\newtheorem{definition}[theorem]{Definition}
\newtheorem{example}[theorem]{Example}
\newtheorem{remark}[theorem]{Remark}

\newcommand{\N}{\mathbb{N}}
\newcommand{\R}{\mathbb{R}}
\newcommand{\eps}{\varepsilon}
\renewcommand{\m}{\mathfrak{m}}
\newcommand{\A}{\mathscr{A}}
\newcommand{\C}{\mathscr{C}}
\newcommand{\RP}{\mathsf{P}}
\newcommand{\RVar}{\mathsf{Var}}
\newcommand{\RBV}{\mathsf{BV}}
\newcommand{\RW}{\mathsf{W}}
\newcommand{\RD}{\mathsf{D}}
\renewcommand{\phi}{\varphi}
\renewcommand{\rho}{\varrho}
\renewcommand{\theta}{\vartheta}

\DeclareMathOperator{\supp}{supp}
\DeclareMathOperator*{\esssup}{ess\,sup}
\DeclareMathOperator*{\essinf}{ess\,inf}
\DeclareMathOperator{\Lip}{Lip}
\DeclareMathOperator{\loc}{loc}	
\DeclareMathOperator{\Var}{Var}	
\DeclareMathOperator{\AVR}{AVR}
\DeclareMathOperator{\Divv}{div}
\renewcommand{\div}{\Divv}
\newcommand{\de}{\mathrm{d}}				    
       
\DeclarePairedDelimiter{\set}{\{}{\}}

\mathchardef\ordinarycolon\mathcode`\:
\mathcode`\:=\string"8000
\begingroup \catcode`\:=\active
  \gdef:{\mathrel{\mathop\ordinarycolon}}
\endgroup

\allowdisplaybreaks

\begin{document}

\title[Cheeger problem in measure spaces]{The Cheeger problem\\ in abstract measure spaces}

\author[V.~Franceschi]{Valentina Franceschi}
\address[V.~Franceschi]{Dipartimento di Matematica ``Tullio Levi Civita'', Università di Padova, via Trieste 63, 35121 Padova (PD), Italy}
\email{valentina.franceschi@unipd.it}

\author[A.~Pinamonti]{Andrea Pinamonti}
\address[A.~Pinamonti]{Dipartimento di Matematica, Università di Trento, via Sommarive 14, 38123 Povo (TN), Italy}
\email{andrea.pinamonti@unitn.it}

\author[G.~Saracco]{Giorgio Saracco}
\address[G.~Saracco]{Dipartimento di Matematica e Informatica ``Ulisse Dini'', Universit\`a di Firenze, viale Morgagni 67/A, 50134 Firenze (FI), Italy}
\email{giorgio.saracco@unifi.it}

\author[G.~Stefani]{Giorgio Stefani}
\address[G.~Stefani]{Scuola Internazionale Superiore di Studi Avanzati (SISSA), via Bonomea 265, 34136 Trieste (TS), Italy}
\email{gstefani@sissa.it {\normalfont or} giorgio.stefani.math@gmail.com}

\date{\today}



\begin{abstract}
We consider non-negative $\sigma$-finite measure spaces coupled with a proper functional $P$ that plays the role of a perimeter. We introduce the Cheeger problem in this framework and extend many classical results on the Cheeger constant and on Cheeger sets to this setting, requiring minimal assumptions on the pair measure space-perimeter. Throughout the paper, the measure space will never be asked to be metric, at most topological, and this requires the introduction of a suitable notion of Sobolev spaces, induced by the coarea formula with the given perimeter.
\end{abstract}
 \hspace{-2cm}
 {
 \begin{minipage}[t]{0.6\linewidth}
 \begin{scriptsize}
 \vspace{-3cm}
 This is a pre-print of an article published in \emph{J.\ London Math.\ Soc}. The final authenticated version is available online at: \href{https://doi.org/10.1112/jlms.12840}{https://doi.org/10.1112/jlms.12840}
 \end{scriptsize}
\end{minipage} 
}

\maketitle

\section{Introduction}

In the Euclidean framework, the \emph{Cheeger constant} of a given set $\Omega\subset \mathbb{R}^n$ is defined as
\[
h(\Omega)=\inf\set*{\frac{P(E)}{\mathscr{L}^n(E)}\,:\, E\subset \Omega,\, \mathscr{L}^n(E)>0},
\]
where $\mathscr{L}^n(E)$ and $P(E)$, respectively, denote the $n$-dimensional Lebesgue measure of $E$ and the variational perimeter of $E$. The constant was first introduced by Jeff Cheeger in a Riemannian setting as a way to bound from below the first eigenvalue of the Laplace--Beltrami operator~\cite{Cheeger70}. The argument proposed is sound and robust, as noticed even earlier by Maz'ya~\cites{Maz62,Maz62-2} (an English translation is available in~\cite{Grigo98}).

In the past decades, the Cheeger constant has been extensively studied in view of its many applications: the lower bounds on the first eigenvalue of the Dirichlet $p$-Laplacian operator~\cite{KF03} and the equivalence of such an inequality with Poincar\'e's one~\cite{Milman09} (up to some convexity assumptions); the relation with the torsion problem~\cites{BE11, BEMP16}; the existence of sets with prescribed mean curvature~\cites{ACC05, LS20}; the existence of graphs with prescribed mean curvature~\cites{Giu78, LS18}; the reconstruction of noisy images~\cites{CC07, CL97, DAG09, ROF92}; and the minimum flow-maximum cut problem~\cites{Str83, Gri06} and its medical applications~\cite{AT06}. In addition, the Cheeger constant has been employed in elastic-plastic models of plate failure~\cite{Kel80} and (its Euclidean-weighted variant) has found applications to Bingham fluids~\cite{HG23} and landslide models~\cite{ILR05}. Moreover, the Cheeger constant of a square has been recently used to provide an elementary proof of the Prime Number Theorem~\cite{Ban64}. For more literature and a general overview of the problem, we refer the interested reader to the two surveys~\cites{Leo15, Parini11}. 

Because of its numerous applications, several authors have been drawn to the subject and started to investigate the constants and the above-mentioned links with other problems in several frameworks: weighted Euclidean spaces~\cites{ABCMP21, BP21, LT19, Saracco18}; finite-dimensional Gaussian spaces~\cites{CMN10, JS21}; anisotropic Euclidean and Riemannian spaces~\cites{ABT15, BNP01, CCMN08, KN08}; the fractional perimeter~\cite{BLP14} or non-singular non-local perimeter functionals~\cite{MRT19}; Carnot groups~\cite{Montefalcone13};  $\mathsf{RCD}$-spaces~\cites{DM21, DMS21}; and lately smooth metric-measure spaces~\cite{AAA21}.

In the settings mentioned above, the proofs mostly follow those available in the usual Euclidean space. In the present paper, we are interested in pinpointing the minimal assumptions needed on the space and on the perimeter functional in order to establish the fundamental properties of Cheeger sets, as well as the links to other problems. In the following, we shall be interested in non-negative $\sigma$-finite measure spaces endowed with a perimeter functional satisfying some suitable assumptions.

Our approach fits into a broader current of research that has gained popularity in the past decade, aiming to study variational problems, well-known in the Euclidean setting, in general spaces under the weakest possible assumptions. Quite often, the ambient space is a (metric) measure spaces. For example, such a general point of view has been adopted for the variational mean curvature of a set~\cite{BM16}, for shape optimization problems~\cite{BV13}, for Anzellotti--Gauss--Green formulas~\cite{GM21}, for the total variation flow~\cites{BCP22, BKP21} and, very recently, for the existence of isoperimetric clusters~\cite{NPST21}.

\subsection{Structure of the paper and results} 

In \cref{sec:PM_spaces}, we introduce the basic setting of \emph{perimeter $\sigma$-finite measure spaces}, that is, non-negative $\sigma$-finite measure spaces $(X, \A, \m)$ endowed with a proper functional $P\colon\A\to[0,+\infty]$ possibly satisfying suitable properties~\ref{prop:empty}--\ref{prop:complement_set} that we shall require from time to time.

A considerable effort goes toward defining $BV$ functions in measure spaces, where a metric is not available. Indeed, usually, the perimeter functional is \emph{induced} by the metric. In our setting, instead, only a perimeter functional is at disposal, so we use it to \emph{define} $BV$ functions by defining the total variation via the coarea formula with the given perimeter. 

To properly define Sobolev functions, we need a slightly richer structure, requiring the measure space to be endowed with a topology, and the perimeter functional $P(\cdot)$ to arise from a \emph{relative} (with respect to open sets $A$) perimeter functional $P(\,\cdot\,; A)$.
By using the relative perimeter, we refine the notion of $BV$ function by requiring that the variation is a finite measure.
This, in turn, allows us to define $W^{1,1}$ functions as $BV$ functions whose variation is absolutely continuous with respect to the reference measure. 
Afterward, via relaxation, we can define $W^{1,p}$ functions for any $p\in(1,+\infty)$. For a more detailed overview, we refer the reader to \cref{ssec:intro_1_eigen} and \cref{ssec:intro_p_eigen}.

Once the general framework is set, we then shall start to tackle the problem of our interest.

\subsubsection{Definition and existence} 
In \cref{sec:cheeger_sets_in_PM_spaces}, we define the Cheeger constant of a set $\Omega$ in terms of the ratio of the perimeter functional and the measure the space is endowed with. Actually, in a more general vein similar to that of~\cite{CL19}, we shall define the \emph{$N$-Cheeger constant} as
\[
h_N(\Omega)
=\inf\set*{\sum_{i=1}^N\frac{P(\mathcal{E}(i))}{\m(\mathcal{E}(i))} :\ \mathcal{E}=\{\mathcal{E}(i)\}_{i=1}^N\subset \Omega\ \text{is an $N$-cluster}},
\]
where, as usual, an \emph{$N$-cluster} is an $N$-tuple of pairwise disjoint subsets of $\Omega$, called \emph{chambers}, each of which with positive finite measure and finite perimeter.

In \cref{res:existence_N-Cheeger}, we provide a general existence result. Unsurprisingly, the key assumptions on the perimeter are the lower-semicontinuity and the compactness of its sublevels with respect to the $L^1$ norm, besides an isoperimetric-type property that prevents minimizing sequences to converge toward sets with null $\m$-measure.

Further, we provide inequalities between the $N$-Cheeger and $M$-Cheeger constants and prove some basic properties of $N$-Cheeger sets, with a particular attention to the case $N=1$.

\subsubsection{Link to sets with prescribed mean curvature}
In \cref{sec:PMC}, we introduce the notion of \emph{$P$-mean curvature} in the spirit of~\cite{BM16}. With this notion at our  disposal, we show that any $1$-Cheeger set has $h_1(\Omega)$ as one of its $P$-mean curvatures, see \cref{cor:variational_curvature_Cheeger}.
An analogous result holds for the chambers of an $N$-cluster minimizing $h_N(\Omega)$, see \cref{cor:variational_curvature_N-Cheeger}.

In \cref{prop:min_PMC}, we investigate the link between $h_1(\Omega)$ and the existence of non-trivial minimizers of the \emph{prescribed $P$-curvature functional}
\[
\mathcal{J}_\kappa[F]=P(F) - \kappa \m(F),
\]
where $\kappa$ is a fixed positive constant, among subsets $F\subset \Omega$. Such a functional, again requiring the lower-semicontinuity and the $L^1$ compactness of sublevel sets of the perimeter, has minimizers. If, additionally, one assumes that the perimeter functional satisfies $P(\emptyset)=0$, then $h_1(\Omega)$ acts as a threshold for the existence of non-trivial minimizers, that is, for $\kappa<h_1(\Omega)$  negligible sets are the only minimizers, while for $\kappa>h_1(\Omega)$  non-trivial minimizers exist.

\subsubsection{Link to the first eigenvalue of the Dirichlet $1$-Laplacian}\label{ssec:intro_1_eigen}

In the Euclidean space, one defines the \emph{first eigenvalue of the Dirichlet $1$-Laplacian} in a variational way as the infimum
\begin{equation}\label{eq:stegosauro}
\lambda_{1,1}(\Omega) = \inf\set*{\frac{\displaystyle\int_\Omega |\nabla u|\,\de x}{\|u\|_1}: u\in \mathrm{C}^1_c(\Omega), \ \|u\|_1>0}.
\end{equation}
In \cref{sec:realtion_with_1-eigen}, we investigate the relation between the $1$-Cheeger constant and a suitable reformulation of the constant~$\lambda_{1,1}(\Omega)$ in our abstract context.

In the Euclidean setting~\cite{KF03} and, actually, in the more general anisotropic central-symmetric Euclidean setting~\cite{KN08}, the constant $\lambda_{1,1}(\Omega)$ coincides with $h_1(\Omega)$ provided that the boundary of the set $\Omega$ is sufficiently smooth (e.g., Lipschitz regular). In particular, one can equivalently consider either smooth functions or $BV$-regular functions. Moreover, because of the smoothness of the boundary of $\Omega$, it holds that $BV(\Omega)=BV_0(\Omega)$, where
\[
BV_0(\Omega) = \set*{u\in BV(\R^n): u=0\ \text{a.e.~on}\ \R^n\setminus \Omega},
\]
see, e.g.,~\cite{CL19}*{Rem.~1.1} or~\cite{CC07}. Thus, under some regularity assumptions on $\Omega$, one can equivalently restate the problem in~\eqref{eq:stegosauro} as
\begin{equation}\label{eq:triceratopo}
\lambda_{1,1}(\Omega) = \inf\set*{\frac{\displaystyle\int_{\mathbb{R}^n} \de |Du|}{\|u\|_1} : u\in BV_0(\R^n),\ \|u\|_1>0}.
\end{equation}

On a general set $\Omega$ in the Euclidean space, the infimum in~\eqref{eq:triceratopo} is less than or equal to that in~\eqref{eq:stegosauro}, since one only has the inclusion $BV(\Omega)\subset BV_0(\Omega)$.

In a (possibly non-metric) perimeter-measure space, the constant $\lambda_{1,1}(\Omega)$ has to be suitably defined, since neither a notion of derivative (needed to state~\eqref{eq:stegosauro}) nor integration-by-parts formulas (needed to define $BV$ functions and thus state~\eqref{eq:triceratopo}) are at disposal.
To overcome this difficulty, we adopt the usual point of view~\cites{Vis90,Vis91,CGL10} and define the \emph{total variation} of a function via the (generalized) coarea formula
\begin{equation}\label{eq:trex}
\Var(u) = \int_\R P(\{u>t\})\,\de t,
\end{equation}
provided that the function $t\mapsto P(\{u>t\})$ is $\mathscr L^1$-measurable, and define the relevant $BV$ space as that of those $L^1$ functions with finite total variation. For more details, we refer the reader to our \cref{ssec:BV_measure_spaces}.

This approach allows us to consider problem~\eqref{eq:triceratopo} without any underlying metric structure.
In addition, no regularity of the set $\Omega$ is required, since there is no need for the  problem~\eqref{eq:triceratopo} to be equivalent to its regular counterpart~\eqref{eq:stegosauro} that, in the present abstract framework, cannot be even formally stated.

With this notion of total variation at hand, we prove that the constant $\lambda_{1,1}(\Omega)$ coincides with the $1$-Cheeger constant $h_1(\Omega)$ under minimal assumptions on the perimeter, that is, we require that the perimeter of negligible sets and of the whole space is zero, the perimeter is lower-semicontinuous with respect to the $L^1$ norm, and that the perimeter of a set coincides with that of its complement set, see \cref{res:relation_with_first_1_eigen}. Moreover, we prove some inequalities relating the $N$-Cheeger constant $h_N(\Omega)$ with a cluster counterpart of~\eqref{eq:triceratopo}. 
As observed in \cref{rem:relation_with_first_1_eigen_no_complement}, if one slightly modifies the  definition of $\lambda_{1,1}(\Omega)$ by considering non-negative functions as the only competitors, then one can obtain the relation with the Cheeger constant even for perimeter functionals which are not symmetric with respect to the complement-set operation.

\subsubsection{Link to the Dirichlet $p$-Laplacian and the $p$-torsion}\label{ssec:intro_p_eigen}

In the Euclidean space, the $1$-Cheeger constant comes into play in estimating some quantities related to the Laplace equation and to the torsional creep equation. More precisely, it provides lower bounds on the \emph{first eigenvalue of the Dirichlet $p$-Laplacian} for $p>1$ and to the $L^1$ norm of the \emph{$p$-torsional creep function}. In \cref{sec:p_eigenvalue}, we extend these results to our more general framework.

Both these problems require an extensive preliminary work to define Sobolev spaces in our general (non-metric) context. In order to do so we need a little more structure on the perimeter-measure space: we require it to be endowed with a topology, we require the class of measurable sets to be that of Borel sets, and we require the perimeter $P(\cdot)$ to stem from a relative perimeter when evaluated relatively to the whole space $X$.

We here quickly sketch how we construct these Sobolev spaces, and we refer the interested reader to \cref{ssec:relative_perimeter}. A relative perimeter functional allows, again via the relative coarea formula in a similar fashion to~\eqref{eq:trex}, to define the relative variation of an $L^1$ function $u$ with respect to a measurable set. When this happens to define a measure, we shall say that the function is in $\RBV(X, \m)$, and this extends the notion briefly discussed in \cref{ssec:intro_1_eigen} and formally introduced in \cref{ssec:BV_measure_spaces}. When this measure happens to be absolutely continuous with respect to $\m$, we shall say that the function is in $\RW^{1,1}(X,\m)$ and that the density of the measure with respect to $\m$ is the \emph{$1$-slope} of $u$. Via approximation arguments, one can then define the \emph{$p$-slope} of a function and the associated $\RW^{1,p}(X,\m)$ spaces. In turn, the approximation properties allow to define the Sobolev space $\RW^{1,p}_0(\Omega, \m)$, refer to \cref{def:sobolev_zero}.

Summing up, Sobolev spaces can be built as induced by a relative perimeter on the topological perimeter-measure space. Once this notion is available, one can define the first eigenvalue of the Dirichlet $p$-Laplacian for $p>1$ in an analogous manner to the standard, Euclidean one. In the classical setting, similarly to~\eqref{eq:stegosauro}, one defines
\begin{equation}\label{eq:carnotauro}
\lambda_{1,p}(\Omega) = \inf\set*{\frac{\displaystyle\int_{\Omega} |\nabla u|^p \de x}{\|u\|^p_p} : u\in \mathrm{C}^1_c(\Omega),\ \|u\|_p>0}.
\end{equation}
In our setting we cannot directly consider~\eqref{eq:carnotauro}, since no notion of derivative is available. However, the natural space of competitors of such a problem is the classical space of $W^{1,p}_0(\Omega)$ functions, and we do have an analogous notion of Sobolev space at our disposal, and thus, such a way is viable.

In Euclidean settings~\cites{KF03,KN08}, it is known that the inequality
\begin{equation*}
\left(\frac{h_1(\Omega)}{p}\right)^p \le \lambda_{1,p}(\Omega)
\end{equation*}
holds. In \cref{thm:CheegerIn} and \cref{res:relation_p_eigen_cheeger}, we prove that this inequality naturally extends to our general framework, provided that the relative perimeter satisfies some general assumptions.

Finally, we recall that the $p$-torsional creep function is the solution of the PDE with homogeneous Dirichlet boundary datum
\begin{equation}
    \label{eq:intro_torsional}
\begin{cases}
\begin{aligned}
-\Delta_p u &= 1, \qquad &&\text{in $\Omega$,}
\\
u&=0, &&\text{on $\partial \Omega$,}
\end{aligned}
\end{cases}
\end{equation}
where $-\Delta_p $ is the $p$-Laplace operator. It is known~\cite{BE11} that the solution $w_p$ of the PDE~\eqref{eq:intro_torsional} satisfies
\begin{equation}
    \label{eq:intro_torsional_est}
h_1(\Omega) \le p \left(\frac{\m(\Omega)}{\|w_p\|_1}\right)^{\frac{p-1}{p}}.
\end{equation}
As usual, we cannot directly consider~\eqref{eq:intro_torsional}, but we can work with the underlying Euler--Lagrange energy among functions in the Sobolev spaces we defined. In particular, we can prove that minimizers of the energy, if they exist, satisfy~\eqref{eq:intro_torsional_est} up to a slightly worse prefactor of $p^{1+\frac 1p}$, refer to \cref{thm:torsional}, provided that the relative perimeter satisfies some very general properties.

\subsubsection{Examples} In the last section of the paper, we collect several examples of spaces that meet our hypotheses. In particular, our very general approach basically covers all results known so far about the existence of Cheeger sets in finite-dimensional spaces, and the relation of the constant with the first eigenvalue of the Dirichlet $p$-Laplacian in numerous contexts. In some of the frameworks presented in \cref{sec:examples}, the results are new, up to our knowledge. 

Unfortunately, our approach does not cover the case of the infinite-di\-men\-sio\-nal Wiener space. In this case, one can suitably define the Cheeger constant and prove the existence of Cheeger sets. Nonetheless, this requires \textit{ad hoc} notions of $BV$ function and of perimeter that are quite different from the ones adopted in the present paper. We refer the interested reader to \cite{CMN10}*{Sect.~6} for a more detailed exposition about this specific framework.

\subsection*{Acknowledgements} 
The authors would like to thank Alessandro Carbotti, Michele Miranda and Diego Pallara for some useful discussions, and Nicol\`o De Ponti for several precious comments and for pointing out various references. Moreover, the authors would like to thank the reviewers for their thoughtful comments and efforts towards improving our manuscript.

The authors are members of the Istituto Nazionale di Alta Matematica (INdAM), Gruppo Nazionale per l'Analisi Matematica, la Probabilit\`a e le loro Applicazioni (GNAMPA).

The first-named, third-named and fourth-named author have received funding from INdAM under the INdAM--GNAMPA Project 2020 \textit{Problemi isoperimetrici con anisotropie} (n.~prot.~U-UFMBAZ-2020-000798 15-04-2020).

The first-named and the fourth-named authors have received funding from INdAM under the INdAM--GNAMPA 2022 Project \textit{Analisi
geometrica in strutture subriemanniane}, identification number CUP\_E55\-F22\-00\-02\-70\-001. 

The second-named author has received funding from INdAM under the INdAM--GNAMPA Project 2022 \textit{Problemi al bordo e applicazioni geometriche}, identification number CUP\_E55\-F22\-00\-02\-70\-001.

The third-named author has received funding from INdAM under the IN\-dAM--GNAM\-PA Project 2022 \textit{Stime ottimali per alcuni funzionali di forma}, identification number CUP\_E55\-F22\-00\-02\-70\-001.

The third-named author has received funding from Universit\`a di Trento (UNITN) under the Starting Grant Giovani Ricercatori 2021 project \textit{WeiCAp}, identification number CUP\_E65\-F21\-00\-41\-60\-001.

The fourth-named author has received funding from INdAM under the INdAM--GNAMPA Project 2023 \textit{Problemi variazionali per funzionali e operatori non-locali}, identification number CUP\_E53\-C22\-001\-930\-001.

The fourth-named author has received funding from the European Research Council (ERC) under the European Union Horizon 2020 research and innovation program (grant agreement No.~945655).

The present research has been started during a visit of the first-named and fourth-named authors at the Mathematics Department of the University of Trento. 
The authors have worked on the present paper during the \textit{XXXI Convegno Nazionale di Calcolo delle Variazioni} held in Levico Terme (Italy), supported by the the Mathematics Department of the University of Trento and the Centro Internazionale per la Ricerca Matematica (CIRM). The authors wish to thank these institutions for their support and kind hospitality.

\section{Perimeter-measure spaces}
\label{sec:PM_spaces}

The basic setting is that of non-negative $\sigma$-finite measure spaces $(X,\A, \m)$. We set that, for any $A,B\in\A$, by $A\subset B$ we mean that $\m(A\setminus B)=0$. We also let $L^0(X,\m)$ be the vector space of $\m$-measurable functions and, for $p\geq 1$, we let $L^p(X,\m)$ be the usual space of $p$-integrable functions, that is,
\[
L^p(X,\m)=\set*{u\in L^0(X,\m) : \int_X|u|^p\;\de \m<+\infty}.
\]
As usual, we identify $\m$-measurable functions coinciding $\m$-a.e.\ on~$X$.
In case $X$ is endowed with a topology $\mathscr T\subset\mathscr P(X)$, we let $\mathscr B(X)$ be the Borel $\sigma$-algebra generated by $\mathscr T$ and, in this case, we shall assume that $\mathscr A=\mathscr B(X)$.

\subsection{Perimeter functional}
\label{ssec:perimeter_functional}

In the same spirit of~\cite{BM16}*{Sect.~3}, we introduce the following definition.
\begin{definition}
    A \emph{perimeter functional} $P(\cdot)$ is any map 
    \begin{equation}
\label{eq:def_perimeter}
P\colon\A\to[0,+\infty],
\end{equation}
which is proper, i.e., $P(A)<+\infty$ for some $A\in\mathscr A$.
In this case, we call $(X,\mathscr A,\m,P)$ a \emph{perimeter-measure space}.
\end{definition}
Throughout the paper we will assume that the perimeter will satisfy some of the following properties:
\begin{enumerate}[label=(P.\arabic*)]
\item\label{prop:empty}
$P(\emptyset)=0$;
\item\label{prop:space}
$P(X)=0$;
\item\label{prop:sub-mod}
$P(E\cap F)+P(E\cup F)\le P(E)+P(F)$
for all $E,F\in\A$;
\item\label{prop:lsc}
$P$ is lower-semicontinuous with respect to the $L^1(X,\m)$ convergence;
\item\label{prop:compact} 
for any $\Omega\in\A$ with $\m(\Omega)<+\infty$, the family
\begin{equation*}
\set*{\chi_E : E\in\A,\ E\subset\Omega,\ P(E)\le c}
\end{equation*}
is compact in $L^1(X,\m)$ for all $c\ge0$;
\item\label{prop:isoperim}
there exists a function $f\colon(0,+\infty)\to (0,+\infty)$ such that
\[
\lim\limits_{\eps\to0^+}f(\eps)=+\infty
\]
with the following property: if $\eps>0$ and $E\in\A$ with $\m(E)\le\eps$, then $P(E)\ge f(\eps)\,\m(E)$;
\item\label{prop:complement_set}
$P(E)=P(X\setminus E)$
for all $E\in\A$.
\end{enumerate}

Assuming property~\ref{prop:complement_set} true, properties~\ref{prop:empty} and~\ref{prop:space} become equivalent. 
Throughout the paper, we often refer to~\ref{prop:isoperim} as an \emph{isoperimetric property}. 
Notice that, in case an isoperimetric inequality $P(E)\geq C \m(E)^{\frac{Q-1}{Q}}$ holds true for suitable $Q>1$ and $C>0$, and for all $E\in\mathcal{A}$ with $\m(E)<+\infty$, then \ref{prop:isoperim} clearly follows. Depending on the situation, it could be more convenient to prove \ref{prop:isoperim} directly or to rely on a finer isoperimetric-type inequality, see \cref{sec:examples}.
We remark that all the properties listed above will appear every now and then throughout the paper, but they are \emph{not enforced} throughout---every statement will precisely contain the bare minimum for its validity.

\begin{remark}[$P$ is invariant under $\m$-negligible modifications]
\label{rem:negligible_modifications}
Let property \ref{prop:lsc} be in force.
If $A,B\in\A$ are such that $\m(A\bigtriangleup B)=0$, then $P(A)=P(B)$. To see this, consider any measurable set $E$ and any $\m$-negligible set $N$, look at the constant sequence $\{E\cup N\}_k$ converging to $E$ in $L^1(X,\m)$, and at the constant one $\{E\}_k$ converging to $E\cup N$ and exploit \ref{prop:lsc}.
\end{remark}

\begin{remark}
Let property \ref{prop:isoperim} be in force. 
If $P(E)=0$, then the set $E$ is $\m$-negligible, that is, $\m(E)=0$. 
Conversely, if $\m(E)>0$, then $P(E)\in(0,+\infty]$. 
Thus, property \ref{prop:isoperim} says that the only sets with finite measure that could possibly have zero perimeter are $\m$-negligible sets.
Moreover, if properties \ref{prop:empty} and \ref{prop:lsc} are in force as well, then $\m$-negligible sets have zero perimeter, thanks to~\cref{rem:negligible_modifications}.
\end{remark}

\subsection{Variation and \texorpdfstring{$BV$}{BV} functions}
\label{ssec:BV_measure_spaces}

We define the \emph{variation} of a function $u\in L^0(X,\m)$ as
\begin{equation}
\label{eq:def_var_by_coarea}
\Var(u) =
\begin{cases}
\displaystyle
\int_\R P(\{u> t\})\, \de t,
&\text{if $t\mapsto P(\{u> t\})$ is $\mathscr L^1$-measurable,}
\\[5mm]
+\infty, 
&\text{otherwise.}
\end{cases}
\end{equation}
With this notation at hand, we let
\begin{equation}
\label{eq:def_BV_space_total}
BV(X,\m)= \set*{u\in L^1(X, \m): \Var(u)<+\infty}
\end{equation}
be the set of $L^1$ functions with \emph{bounded variation}.

We begin with the following result, proving that assuming the validity of properties~\ref{prop:empty} and~\ref{prop:space}, the variation coincides with the perimeter functional on characteristic functions.
 
\begin{lemma}[Total variation of sets]
\label{res:tot_var_of_sets}
Let properties~\ref{prop:empty} and \ref{prop:space} be in force. 
If $E\in\mathscr A$, then 
$\Var(\chi_E)=P(E)$.
\end{lemma}

\begin{proof}
By definition, \ref{prop:empty}, and \ref{prop:space}, the function 
\begin{equation*}
t\mapsto P(\set*{\chi_E>t})
=
\begin{cases}
P(X), & t\le0,
\\
P(E), & 0<t\le1,
\\
P(\emptyset), & t>1,
\end{cases}
\end{equation*}
is $\mathscr L^1$-measurable, so that
\begin{equation*}
\Var(\chi_E)=\int_{\R}P(\set*{\chi_E>t})\,\de t
=
\int_0^1 P(E)\,\de t = P(E),
\end{equation*}
in virtue of \ref{prop:empty} and of \ref{prop:space}.
\end{proof}

\begin{remark}
\label{rem:zero_is_BV}
As an immediate consequence of \cref{res:tot_var_of_sets}, if~\ref{prop:empty} and~\ref{prop:space} are in force, then $\Var\colon L^0(X,\m)\to[0,+\infty]$ is a proper functional and $\chi_E\in BV(X,\m)$ whenever $E\in\mathscr A$ is such that $\m(E)<+\infty$ and $P(E)<+\infty$.
In particular, $0\in BV(X,\m)$ with $\Var(0)=0$.
\end{remark}

The following result rephrases~\cite{CGL10}*{Prop.~3.2} in the present context. 

\begin{lemma}[Basic properties of total variation]
\label{res:basic_props_var}
The following hold:
\begin{enumerate}[label=(\roman*)]

\item\label{item:prop_var_scaling} 
$\Var(\lambda u)=\lambda\Var(u)$ for all $\lambda>0$ and $u\in L^0(X,\m)$;

\item\label{item:prop_var_add_const}
$\Var(u+c)=\Var(u)$ for all $c\in\R$ and $u\in L^0(X,\m)$;

\item\label{item:prop_var_constant}
if~\ref{prop:empty} and~\ref{prop:space} are in force, then $\Var(c)=0$ for all $c\in\R$;

\item\label{item:prop_var_lsc}
if~\ref{prop:lsc} is in force, then $\Var\colon L^1(X,\m)\to[0,+\infty]$ is lower-semi\-con\-tin\-u\-ous with respect to the (strong) convergence in $L^1(X,\m)$.

\end{enumerate}
\end{lemma}

\begin{proof}
The proofs of the first three points are natural consequences of the definition.

\vspace{1ex}

\textit{Proof of~\ref{item:prop_var_lsc}}.
Let $u_k,u\in L^1(X,\m)$ be such that $u_k\to u$ in $L^1(X,\m)$ as $k\to+\infty$. 
Without loss of generality, we can assume that
\[
\liminf_{k\to+\infty}\Var(u_k)<+\infty,
\] 
so that, up to possibly passing to a subsequence (which we do not relabel for simplicity), we have $\Var(u_k)<+\infty$ for all $k\in\N$.
Following~\cite{Maggi}*{Rem.~13.11}, one has 
\begin{equation*}
\|u_k-u\|_1
=
\int_\R\m(\set*{u_k>t}\bigtriangleup\set*{u>t})\,\de t,
\end{equation*}
thus, we immediately deduce that $\chi_{\set{u_k>t}}\to\chi_{\set{u>t}}$ in $L^1(X,\m)$ as $k\to+\infty$ for $\mathscr L^1$-a.e.\ $t\in\R$.
Thanks to property~\ref{prop:lsc}, we have that 
\begin{equation*}
P(\set*{u>t})
\le
\liminf_{k\to+\infty}
P(\set*{u_k>t})
\end{equation*}
for $\mathscr L^1$-a.e.\ $t\in\R$, and the map $t\mapsto P(\set*{u>t})$ is also $\mathscr L^1$-measurable. Therefore, by Fatou's Lemma, we conclude that 
\begin{align*}
\Var(u)
&=
\int_\R P(\set*{u>t})\,\de t
\le
\int_\R\liminf_{k\to+\infty}
P(\set*{u_k>t})\,\de t
\\
&\le
\liminf_{k\to+\infty}
\Var(u_k)
<+\infty,
\end{align*}
proving~\ref{item:prop_var_lsc}.
\end{proof}

The following result, which can be proved as in~\cite{CGL10} up to minor modifications, states that the variation functional is convex as soon as the perimeter functional is sufficiently well-behaved. 

\begin{proposition}[Convexity of variation]
\label{res:var_convex}
Let properties~\ref{prop:empty}, \ref{prop:space}, \ref{prop:sub-mod} and~\ref{prop:lsc} be in force. 
Then, 
$
\Var\colon L^1(X,\m)\to[0,+\infty]
$
is convex.
As a consequence, $BV(X,\m)$ is a convex cone in $L^1(X,\m)$.
\end{proposition}

\subsection{Relative perimeter and relative variation}

\label{ssec:relative_perimeter}

In this subsection, we assume that the set~$X$ is endowed with a topology~$\mathscr T$ such that $\mathscr A=\mathscr B(X)$, the Borel $\sigma$-algebra generated by $\mathscr T$.
\begin{definition}
    A \emph{relative perimeter functional} $\RP$ is any map 
   \begin{equation}
\label{eq:def_relative_perimeter}
\RP\colon\mathscr B(X)\times\mathscr B(X)\to[0,+\infty].
\end{equation}
\end{definition}
Throughout the paper, we will assume that a relative perimeter will satisfy some of the following properties:
\begin{enumerate}[label=(RP.\arabic*)]
\item\label{propR:empty}
$\RP(\emptyset;A)=0$ 
for all $A\in\mathscr T$;
\item\label{propR:space}
$\RP(X;A)=0$
for all $A\in\mathscr T$;
\item\label{propR:sub-mod}
$\RP(E\cap F;A)+\RP(E\cup F;A)\le \RP(E;A)+\RP(F;A)$
for all $E,F\in\mathscr B(X)$ and $A\in\mathscr T$;

\item\label{propR:lsc}
for each $A\in\mathscr T$, $\RP(\cdot\,;A)$ is lower-semicontinuous with respect to the (strong) convergence in $L^1(X,\m)$.
\end{enumerate}	
We stress that in the properties above, the perimeter is relative to an open set $A$, and not to a general element of the Borel $\sigma$-algebra.

Moreover, following the same idea of \cref{ssec:BV_measure_spaces}, we let 
\begin{equation}
\label{eq:coarea_relative_def}
\RVar(u;A)
=
\begin{cases}
\displaystyle
\int_{\R} \RP(\set*{u>t};A)\,\de t,
&
\text{if}\ t\mapsto \RP(\set*{u>t};A)\ \text{is $\mathscr L^1$-meas.},
\\[5mm]
+\infty,
& \text{otherwise},
\end{cases}
\end{equation}
be the \emph{variation of $u\in L^0(X,\m)$ relative to  $A\in\mathscr B(X)$}.
In analogy with the approach developed in the previous sections, for each $A\in\mathscr T$ one can regard the map 
\begin{equation*}
\RP(\cdot\,;A)\colon\mathscr B(X)\to[0,+\infty]
\end{equation*}
as a particular instance of the perimeter functional introduced in~\eqref{eq:def_perimeter}.
Specifically, we use the notation 
\begin{equation}
\label{eq:def_tot_RP}
P(E)=\RP(E;X),
\qquad
\Var(u)=\RVar(u;X),
\end{equation}
for all $E\in\mathscr B(X)$ and $u\in L^0(X,\m)$ and we consider $P(E)$ as the perimeter of~$E$ in the sense of \cref{ssec:perimeter_functional}. Analogously, $\Var(u)$ stands as the variation of $u$ in the sense of \cref{ssec:BV_measure_spaces}.
Consequently, the space 
\begin{equation*}
BV(X,\m)=\set*{u\in L^1(X,\m):\RVar(u; X)<+\infty}
\end{equation*}
is the space defined in~\eqref{eq:def_BV_space_total}.
 
Below, we rephrase \cref{res:tot_var_of_sets}, \cref{res:basic_props_var} and \cref{res:var_convex} in the present setting. Their proofs are omitted, because they are similar to those already given or referred to.

\begin{lemma}[Relative variation of sets]
\label{res:relative_var_of_sets}
Let properties~\ref{propR:empty} and~\ref{propR:space} be in force. 
If $E\in\mathscr B(X)$, then $\RVar(\chi_E;A)=\RP(E;A)$ for all $A\in\mathscr T$.
\end{lemma}

\begin{lemma}[Basic properties of relative variation]
\label{res:basic_props_varR}
The following hold:
\begin{enumerate}[label=(\roman*)]

\item 
$\RVar(\lambda u;A)=\lambda\RVar(u;A)$ for all $\lambda>0$, $A\in\mathscr T$ and $u\in L^0(X,\m)$;

\item
$\RVar(u+c;A)=\RVar(u;A)$ for all $A\in\mathscr T$, $c\in \R$, and $u\in L^0(X,\m)$;

\item
if~\ref{propR:empty} and~\ref{propR:space} are in force, then $\RVar(c;A)=0$ for all $c\in\R$ and $A\in\mathscr T$;

\item 
if~\ref{propR:lsc} is in force, then, for each $A\in\mathscr T$, the relative variation $\RVar(\cdot\,;A)\colon L^1(X,\m)\to[0,+\infty]$ is lower-semicontinuous with respect to the (strong) convergence in $L^1(X,\m)$.

\end{enumerate}
\end{lemma}

\begin{proposition}[Convexity of relative variation]
\label{res:convex_varR}
Let properties~\ref{propR:empty}, \ref{propR:space}, \ref{propR:sub-mod}  and~\ref{propR:lsc} be in force.
Then, for each $A\in\mathscr T$, the functional $\RVar(\cdot\,;A)\colon L^1(X,\m)\to[0,+\infty]$ is convex.
\end{proposition}

\subsubsection{Variation measure}

We now define the \emph{perimeter} and \emph{variation measures} by rephrasing the validity of the relative coarea formula~\eqref{eq:coarea_relative_def} in a measure-theoretic sense.

\begin{definition}[Perimeter and variation measures]
\label{def:per_var_measures}
We say that a set $E\in\mathscr B(X)$ has \emph{finite perimeter measure} if its relative perimeter
\begin{equation*}
\RP(E;\,\cdot\,)\colon\mathscr B(X)\to[0,+\infty]
\end{equation*} 
defines a finite outer regular Borel measure on~$X$.
We hence say that a function $u\in L^0(X,\m)$ has \emph{finite variation measure} if the set $\set{u>t}$ has finite perimeter for $\mathscr L^1$-a.e.\ $t\in\R$ and its relative variation
\begin{equation*}
\RVar(u;\,\cdot\,)
\colon
\mathscr B(X)\to[0,+\infty)
\end{equation*} 
defines a finite outer regular Borel measure on~$X$.
\end{definition}

Adopting the usual notation, if $E\in\mathscr B(X)$ has finite perimeter measure, then we write $\RP(E;A)=|\RD\chi_E|(A)$ for all $A\in\mathscr B(X)$.
Similarly, if $u\in L^0(X,\m)$ has finite variation measure, then we write $\RVar(u;A)=|\RD u|(A)$ for all $A\in\mathscr B(X)$.

It is worth noticing that \cref{def:per_var_measures} is well posed in the following sense. 
As soon as properties~\ref{propR:empty} and~\ref{propR:space} are in force, if $E\in\mathscr B(X)$ has finite perimeter measure, then $\chi_E\in L^0(X,\m)$ has finite variation measure with $\RVar(\chi_E;\,\cdot\,)=\RP(E;\,\cdot\,)$, since they are outer regular Borel measures on~$X$ agreeing on open sets. This is a simple consequence of \cref{res:relative_var_of_sets}.

By \cref{def:per_var_measures}, if $u\in L^0(X,\m)$ has finite variation measure, then, for each $A\in\mathscr B(X)$, we have $\RVar(u;A)<+\infty$  and thus
\begin{equation*}
t\mapsto \RP(\set{u>t};A)
\in L^1(\R),
\end{equation*}
so that we can write
\begin{equation*}
|\RD u|(A)=\RVar(u;A)=\int_\R \RP(\set*{u>t};A)\,\de t
=
\int_\R  |\RD\chi_{\set*{u>t}}|(A)\, \de t.
\end{equation*}
In more general terms, we get the following extension of the relative coarea formula~\eqref{eq:coarea_relative_def}.
Its proof follows from a routine approximation argument (see~\cite{AFP00book} for instance) and is thus omitted.

\begin{corollary}[Generalized coarea formula]
\label{res:coarea_gen}
If $u\in L^0(X,\m)$ has finite variation measure, then
\begin{equation*}
\int_A\phi\,\de |\RD u|
=
\int_\R\int_A\phi\,\de |\RD\chi_{\set*{u>t}}|\,\de t
\end{equation*}
for all $\phi\in L^0(X,\m)$ and $A\in\mathscr B(X)$. 
\end{corollary}

Keeping the same notation used in the previous sections, we let 
\begin{equation*}
\RBV(X,\m)=\set*{u\in L^1(X,\m) : u\ \text{has finite variation measure}}.
\end{equation*}
Notice that although $\RBV(X,\m)\subset BV(X,\m)$ and $BV(X,\m)$ is a convex cone in $L^1(X,\m)$, the set $\RBV(X,\m)$ may not be a convex cone in $L^1(X,\m)$ as well, since the validity of the implication 
\begin{equation*}
u,v\in \RBV(X,\m)
\implies
u+v\ \text{has finite variation measure}
\end{equation*}
is not automatically granted.
For an example of such a phenomenon, we refer the interested reader to the variation of intrinsic maps between subgroups of sub-Riemannian Carnot groups~\cite{SCV}*{Rem.~4.2}, but we will not enter into the details of this issue because it is out of the scope of the present paper. 

This being said, we introduce the following additional property for the relative perimeter~$\RP$  in~\eqref{eq:def_relative_perimeter} requiring the closure of $\RBV(X,\m)$ with respect to the sum of functions:
\begin{enumerate}[label=(RP.+)]
\item
\label{propR:BV_cone}
$u,v\in\RBV(X,\m)\implies u+v\in \RBV(X,\m)$.
\end{enumerate}

We now outline some consequences of \cref{res:basic_props_varR} and \cref{res:convex_varR}, and leave the simple proofs of these statements to the interested reader, see also the proof of~\cref{res:basic_props_var}. 

\begin{corollary}[Basic properties of variation measure]
\label{res:basic_props_var_meas}
Let properties~\ref{propR:empty}, \ref{propR:space}, \ref{propR:sub-mod}  and~\ref{propR:lsc} be in force.
The following hold:
\begin{enumerate}[label=(\roman*)]

\item 
if $u\in L^0(X,\m)$ has finite variation measure, then $\lambda u$ has finite variation measure, with $|\RD(\lambda u)|=\lambda|\RD u|$, for all $\lambda>0$;

\item
if $u\in L^0(X,\m)$ has finite variation measure, then $u+c$ has finite variation measure, with $|\RD(u+c)|=|\RD u|$, for all $c\in\R$;

\item
constant functions have finite variation measure and $|\RD c|=0$ for all $c\in\R$ (in particular, $0\in\RBV(X,\m)$); 

\item
if $\{u_k\}_{k\in\N}\subset\RBV(X,\m)$ and $u_k\to u$ in $L^1(X,\m)$ as $k\to+\infty$ for some $u\in\RBV(X,\m)$, then 
\begin{equation*}
|\RD u|(A)\le\liminf_{k\to+\infty}|\RD u_k|(A)
\end{equation*}
for all $A\in\mathscr T$;

\item 
if also property~\ref{propR:BV_cone} is in force and $u,v\in\RBV(X,\m)$, then 
\begin{equation*}
|\RD(u+v)|\le|\RD u|+|\RD v|
\end{equation*}
as outer regular Borel measures on~$X$.
\end{enumerate}
\end{corollary}

\subsubsection{Chain rule}

We now establish a chain rule for the variation measure of continuous functions. 
To this aim, we need to assume the following \emph{locality property} of the relative perimeter functional in~\eqref{eq:def_relative_perimeter}: 
\begin{enumerate}[label=(RP.L)]
\item
\label{propR:local}
$E\in\mathscr T\implies \RP(E;A)=0$
for all $A\in\mathscr B(X)$ with $\RP(E;A\cap\partial E)=0$.
\end{enumerate}
Loosely speaking, property~\ref{propR:local} states that, for any open set $E\subset X$, the relative perimeter functional $A\mapsto \RP(E;A)$ is supported (in a measure-theoretic sense) on the topological boundary~$\partial E$ of the set~$E$.

\begin{theorem}[Chain rule]
\label{res:chain_rule}
Let properties~\ref{propR:empty}, \ref{propR:space} and~\ref{propR:local} be in force and let $\phi\in \mathrm{C}^1(\R)$ be a strictly increasing function.
If $u\in \mathrm{C}^0(X)$ has finite variation measure, then also $\phi(u)\in \mathrm{C}^0(X)$ has finite variation measure, with
\begin{equation}
\label{eq:chain_rule}
|\RD\phi(u)|=\phi'(u)|\RD u|
\end{equation}
as finite outer regular Borel measure on $X$.
\end{theorem}

\begin{proof}
Since $\phi$ is strictly increasing, its inverse function $\phi^{-1}\colon\phi(\R)\to\R$ is well defined, continuous and strictly increasing, and we can write
\begin{equation*}
\set*{\phi(u)>t}
=
\begin{cases}
X,
&\text{if}\ t\le\inf\phi(\R),
\\
\set*{u>\phi^{-1}(t)},
\quad &
\text{if}\ t\in\phi(\R),
\\
\emptyset,
&
\text{if}\ t\ge\sup\phi(\R). 
\end{cases}
\end{equation*}
Therefore, the set $\set*{\phi(u)>t}$ has finite perimeter measure for $\mathscr L^1$-a.e.\ $t\in\R$, with
\begin{equation*}
|\RD\chi_{\set*{\phi(u)>t}}|
=
\begin{cases}
|\RD\chi_{\set*{u>\phi^{-1}(t)}}|,
\quad &
\text{if}\ t\in\phi(\R),
\\
0 ,
&\text{if}\ t\notin\phi(\R).
\end{cases}
\end{equation*}
Hence, given $A\in\mathscr B(X)$, we have
\begin{equation*}
t\mapsto|\RD\chi_{\set*{\phi(u)>t}}|(A)
=
|\RD\chi_{\set*{u>\phi^{-1}(t)}}|(A)\,\chi_{\phi(\R)}(t)
\in L^1(\R)
\end{equation*}
and so,
\begin{equation*}
\RVar(\phi(u);A)
=
\int_\R|\RD\chi_{\set*{\phi(u)>t}}|(A)\,\de t
=
\int_{\phi(\R)}
|\RD\chi_{\set*{u>\phi^{-1}(t)}}|(A)\,\de t.
\end{equation*}
Performing a change of variables, we can write 
\begin{equation*}
\int_{\phi(\R)}
|\RD\chi_{\set*{u>\phi^{-1}(t)}}|(A)\,\de t
=
\int_\R|\RD\chi_{\set*{u>s}}|(A)\,\phi'(s)\,\de s.
\end{equation*}
Now, since $u\in \mathrm{C}^0(X)$, we know that $\set*{u>s}\in\mathscr T$ and $\partial\set*{u>s}\subset\set*{u=s}$ for all $s\in\R$. 
Therefore, because of~\ref{propR:local}, we have 
$|\RD\chi_{\set*{u>s}}|(B)=0$ for all $B\in\mathscr B(X)$ such that 
$|\RD\chi_{\set*{u>s}}|(B\cap\set*{u=s})=0$.

Thus, letting $B=A\cap\set*{u\ne s}$, we have that $|\RD\chi_{\set*{u>s}}|(A\cap\set*{u\ne s})=0$ for all $s\in\R$; hence, the following equalities hold
\begin{align*}
\int_\R|\RD\chi_{\set*{u>s}}|(A)\,\phi'(s)\,\de s
&=
\int_\R\phi'(s)\int_A \de |\RD\chi_{\set*{u>s}}|\,\de s
\\
&=
\int_\R\int_A \phi'(u)\,\de |\RD\chi_{\set*{u>s}}|\,\de s.
\end{align*}
By \cref{res:coarea_gen}, we can write
\begin{equation*}
\int_\R\int_A \phi'(u)\,\de |\RD\chi_{\set*{u>s}}|\,\de s
=
\int_A\phi'(u)\,\de |\RD u|,
\end{equation*}
so that, by combining all the above equalities, we conclude that 
\begin{equation*}
\RVar(\phi(u);A)
=
\int_A\phi'(u)\,\de |\RD u|
\end{equation*}
for all $A\in\mathscr B(X)$, proving~\eqref{eq:chain_rule} and completing the proof.  
\end{proof}

\subsubsection{\texorpdfstring{$p$}{p}-slope and Sobolev functions}

\label{subsec:sobolev_1_space}

As customary, we let 
\begin{equation*}
\RW^{1,1}(X,\m)=\set*{u\in \RBV(X,\m) : |\RD u|\ll\m}
\end{equation*}
be the set of \emph{Sobolev $\RW^{1,1}$ functions} on~$X$. 

If $u\in\RW^{1,1}(X,\m)$, then we let $|\nabla u|\in L^1(X,\m)$, $|\nabla u|\ge0$ $\m$-a.e.\ in~$X$, be the \emph{$1$-slope} of~$u$, i.e., the unique $L^1(X,\m)$ function such that 
\begin{equation*}
|\RD u|(A)
=
\int_A|\nabla u|\,\de \m
\qquad
\text{for all}\ A\in\mathscr B(X).
\end{equation*}
From \cref{res:basic_props_var_meas}, we immediately deduce the following simple properties of $1$-slopes of $\RW^{1,1}$ functions. 

\begin{corollary}
[Basic properties of $1$-slope]
\label{res:basic_props_1_slope}
Let properties~\ref{propR:empty}, \ref{propR:space}, \ref{propR:sub-mod}, \ref{propR:lsc}, and~\ref{propR:BV_cone} be in force.
The following hold:
\begin{enumerate}[label=(\roman*)]

\item 
if $u\in\RW^{1,1}(X,\m)$, then $\lambda u\in\RW^{1,1}(X,\m)$, with $|\nabla(\lambda u)|=\lambda|\nabla u|$, for all $\lambda>0$;

\item
$0\in\RW^{1,1}(X,\m)$, with $|\nabla 0|=0$; 

\item 
if $u,v\in\RW^{1,1}(X,\m)$, then $u+v\in\RW^{1,1}(X,\m)$ with 
\begin{equation*}
|\nabla(u+v)|\le|\nabla u|+|\nabla v|.
\end{equation*}
\end{enumerate}
As a consequence, $\RW^{1,1}(X,\m)$ is a convex cone in $L^1(X,\m)$.
\end{corollary}

Having the notion of $1$-slope at our disposal, following the standard approach about slopes (see~\cite{AGS14} for instance), we can introduce the notion of \emph{$p$-relaxed $1$-slope}, for $p\in(1,+\infty)$.

\begin{definition}[$p$-relaxed $1$-slope]
\label{def:p_relaxed_1_slope}
Let $p\in(1,+\infty)$.
We shall say that a function $g\in L^p(X,\m)$ is a \emph{$p$-relaxed $1$-slope} of $u\in L^p(X,\m)$ if there exist a function $\tilde g\in L^p(X,\m)$ and a sequence $\{u_k\}_{k\in\N}\subset\RW^{1,1}(X,\m)\cap L^p(X,\m)$ such that: 
\begin{enumerate}[label=(\roman*)]

\item 
$u_k\conv*{}{} u$ in $L^p(X,\m)$;

\item\label{item:def_relaxed_1_slope_weak_conv}
$|\nabla u_k|\in L^p(X,\m)$ for all $k\in\N$ and $|\nabla u_k|\wkconv*{}{}\tilde g$ weakly in $L^p(X,\m)$;

\item
$\tilde g\le g$ $\m$-a.e.\ in~$X$.

\end{enumerate}
\end{definition}

Clearly, according to \cref{def:p_relaxed_1_slope} and thanks to the sequential compactness of weak topologies, if $\{u_k\}_{k\in\N}\subset\RW^{1,1}(X,\m)\cap L^p(X,\m)$ is such that 
\begin{equation*}
\sup_{k\in\N}\int_X|\nabla u_k|^p\,\de \m<+\infty,
\end{equation*}
then any $L^p(X,\m)$-limit of $\{u_k\}_{k\in\N}$ has at least one $p$-relaxed $1$-slope. Given any $u\in L^p(X,\m)$, we define
\begin{equation*}
\mathsf{Slope}_p(u)
=
\set*{g\in L^p(X,\m): \text{$g$ is a $p$-relaxed $1$-slope of $u$}}.
\end{equation*}
Following the point of view of~\cite{AGS14}, one can prove the following basic properties of $p$-relaxed $1$-slopes that will be useful in the sequel.

\begin{lemma}[Basic properties of $p$-relaxed $1$-slope]
\label{res:basic_props_p_relaxed_1_slope}
Let properties~\ref{propR:empty}, \ref{propR:space}, \ref{propR:sub-mod}, \ref{propR:lsc}, and~\ref{propR:BV_cone} be in force and let $p\in(1,+\infty)$.
The following hold:
\begin{enumerate}[label=(\roman*)]

\item\label{item:basic_props_p_relaxed_1_slope_convex} 
$\mathsf{Slope}_p(u)$ is a convex subset (possibly empty) for all $u\in L^p(X,\m)$;

\item\label{item:basic_props_p_relaxed_1_slope_strong_conv} 
if $u\in L^p(X,\m)$ and $g\in \mathsf{Slope}_p(u)$, then there exist a sequence $\{u_k\}_k\subset \RW^{1,1}(X,\m)\cap L^p(X,\m)$, a sequence $\{g_k\}_k \subset L^p(X,\m)$, and a function $\tilde g\in L^p(X,\m)$, such that $u_k\conv*{}{} u$ and $g_k\conv*{}{} \tilde g$ both in $L^p(X,\m)$, with $|\nabla u_k|\le g_k$ for all $k\in\N$ and $\tilde g\le g$;

\item\label{item:basic_props_p_relaxed_1_slope_lsc_weak}
if $\{u_k\}_k$ and $\{g_k\}$ are sequences in $L^p(X,\m)$, with $g_k\in\mathsf{Slope}_p(u_k)$ for all $k\in\N$, such that $u_k\wkconv*{}{} u$ and $g_k\wkconv*{}{} g$ weakly in $L^p(X,\m)$, then $g\in\mathsf{Slope}_p(u)$.

\end{enumerate}
\end{lemma}

Under the assumptions of the above \cref{res:basic_props_p_relaxed_1_slope}, for each $u\in L^p(X,\m)$, the set $\mathsf{Slope}_p(u)$ is a (possibly empty) closed convex subset of $L^p(X,\m)$, and thus, the following definition is well posed. 

\begin{definition}[Weak $p$-slope]
\label{def:weak_p_slope}
Let $p\in(1,+\infty)$ and let properties~\ref{propR:empty}, \ref{propR:space}, \ref{propR:sub-mod}, \ref{propR:lsc}, and~\ref{propR:BV_cone} be in force.
If $u\in L^p(X,\m)$ is such that $\mathsf{Slope}_p(u)\ne\emptyset$, we let $|\nabla u|_p$ be the element of $\mathsf{Slope}_p(u)$ of minimal  $L^p(X,\m)$ norm and we call it the \emph{weak $p$-slope} of~$u$.
Finally, we let
\begin{equation*}
\RW^{1,p}(X,\m)=
\set*{
u\in L^p(X,\m): \exists|\nabla u|_p\in L^p(X,\m)
}.
\end{equation*}
\end{definition}

Following the same line of~\cites{AGS14}, one can show that the weak $p$-slope can be actually approximated in $L^p(X,\m)$ in the strong sense. 

\begin{corollary}[Strong approximation of weak $p$-slope]
\label{res:strong_approx_weak_p_slope}
Let $p\in(1,+\infty)$ and let properties~\ref{propR:empty}, \ref{propR:space}, \ref{propR:sub-mod}, \ref{propR:lsc}, and~\ref{propR:BV_cone} be in force.
If $u\in\RW^{1,p}(X,\m)$, then there exists a sequence $\{u_k\}_{k}\subset\RW^{1,1}(X,\m)\cap L^p(X,\m)$ such that $|\nabla u_k|\in L^p(X,\m)$ for all $k\in\N$ and 
\begin{equation*}
u_k\conv*{}{} u\ \text{and}\ |\nabla u_k|\conv*{}{}|\nabla u|_p\ \text{both in}\ L^p(X,\m)\ \text{as}\ k\to+\infty.
\end{equation*}
\end{corollary}

\section{Cheeger sets in perimeter-measure spaces}
\label{sec:cheeger_sets_in_PM_spaces}

In this section, we work in a measure space endowed with a perimeter functional as in \cref{ssec:perimeter_functional}.

\subsubsection{\texorpdfstring{$N$}{N}-Cheeger constant and \texorpdfstring{$N$}{N}-Cheeger sets}

We begin by introducing the central notions of the present paper.

\begin{definition}
Let $N\in\mathbb{N}$. An \emph{$N$-cluster} $\mathcal{E}$ is a collection of $N$ measurable sets $\mathcal{E}=\{\mathcal{E}(i)\}_{i=1}^N \subset \A$ satisfying:
\begin{itemize}
    \item $0<\m(\mathcal{E}(i))<+\infty$ for all $i=1,\ldots, N$;
    \item $\m(\mathcal{E}(i)\cap \mathcal{E}(j))=0$ for all $i,j=1,\ldots, N$ with $i\neq j$;
    \item $P(\mathcal{E}(i))<+\infty$ for all $i=1,\ldots, N$.
\end{itemize}
Each of the $\mathcal{E}(i)$, $i=1,\dots,N$, is called a \emph{chamber}.
\end{definition}

\begin{definition}[$N$-admissible set]
\label{def:admissible_set}
Let $N\in\mathbb{N}$. We say that $\Omega\in\A$ is \emph{N-admissible} if there exists an $N$-cluster $\mathcal{E}=\{\mathcal{E}(i)\}_{i=1}^N\subset \Omega$. 
\end{definition}

\begin{remark}\label{rem:Nadm=>Madm}
Let $N\in\mathbb{N}$. Trivially, if $\Omega\in \A$ is $N$-admissible, then it is $M$-admissible for all integers $M\le N$.
\end{remark}

\begin{definition}[$N$-Cheeger constant and $N$-Cheeger sets]
Let $N\in\mathbb{N}$ and let $\Omega\in \mathscr{A}$ be an $N$-admissible set.
The \emph{$N$-Cheeger constant} of $\Omega$ is
\begin{equation*}
h_N(\Omega)=\inf\set*{\sum_{i=1}^N\frac{P(\mathcal{E}(i))}{\m(\mathcal{E}(i))} :\ \mathcal{E}=\{\mathcal{E}(i)\}_{i=1}^N\subset \Omega\ \text{is an $N$-cluster}}.
\end{equation*}
If $\mathcal{C}=\{\mathcal{C}(i)\}_{i=1}^N$ is an $N$-cluster realizing the above infimum, we call it an $N$-Cheeger set (or cluster) of $\Omega$. We let $\C_N(\Omega)$ be the collection of all $N$-Cheeger sets of $\Omega$.
\end{definition}

\begin{remark}
By definition, as $\Omega$ is required to be $N$-admissible, the $N$-Cheeger constant of $\Omega$ is finite. Moreover, by \cref{rem:Nadm=>Madm}, so it is $h_M(\Omega)$ for all integers $M$ such that $M\le N$. We also refer to \cref{prop:ineq_M_N_constants}.
\end{remark}

\subsection{Existence of \texorpdfstring{$N$}{N}-Cheeger sets} 

We prove that the existence of $N$-Cheeger clusters of $\Omega$ is ensured whenever the perimeter functional possesses properties \ref{prop:lsc}, \ref{prop:compact}, and \ref{prop:isoperim}, and the set $\Omega\in \A$ is $N$-admissible with finite $\m$-measure. These requests are not necessary though, as some examples at the end of this section show.

\begin{theorem}
\label{res:existence_N-Cheeger}
Let properties \ref{prop:lsc}, \ref{prop:compact}, and \ref{prop:isoperim} be in force. Let $N\in \mathbb{N}$, and let $\Omega\in\mathscr{A}$ be an $N$-admissible set with $\m(\Omega)\in (0,+\infty)$. 
Then there exists an $N$-Cheeger set of $\Omega$.
\end{theorem}

\begin{proof}
On the one hand, since $\Omega$ is $N$-admissible, there exists an $N$-cluster $\mathcal{E}\subset \Omega$, which immediately implies that $h_N(\Omega)<+\infty$. 
On the other hand, for any $N$-cluster $\mathcal{E}=\{\mathcal{E}(i)\}_{i=1}^N$ of $\Omega$, property~\ref{prop:isoperim} gives
\begin{equation*}
\sum_{i=1}^N\frac{P(\mathcal{E}(i))}{\m(\mathcal{E}(i))}\geq N f(\m(\Omega))\,,
\end{equation*}
hence
\begin{equation*}
    h_N(\Omega)\geq Nf(\m(\Omega))>0\,.
\end{equation*}
Now let $\{\mathcal{E}_k\}_{k\in\N}\subset \Omega$ be a minimizing sequence, i.e.,\ 
\begin{equation*}
\lim_{k\to+\infty}\sum_{i=1}^N\frac{P(\mathcal{E}_k(i))}{\m(\mathcal{E}_k(i))}=h_N(\Omega).
\end{equation*}
Clearly, for any $k\in\mathbb{N}$ sufficiently large and any $i=1,\ldots, N$, we have
\begin{equation*}
    P(\mathcal{E}_k(i))\leq \m(\Omega)\sum_{i=1}^N\frac{P(\mathcal{E}_k(i))}{\m(\mathcal{E}_k(i))}\leq \m(\Omega)(h_N(\Omega)+1)
\end{equation*}
and thus,
\begin{equation*}
    \sup_{k}\set*{\max_i\set*{ P(\mathcal{E}_k(i))}}\leq \m(\Omega)(h_N(\Omega)+1),
\end{equation*}
which is finite, having assumed $\m(\Omega)<+\infty$.

By \ref{prop:lsc} and \ref{prop:compact} (recall also \cref{rem:negligible_modifications}), possibly passing to a subsequence, for each $i=1,\dots, N$, there exists $\mathcal{E}(i)\in\A$ such that $\mathcal{E}(i)\subset\Omega$, with $\m(\mathcal{E}(i))\in[0,\m(\Omega)]$, $P(\mathcal{E}(i))\le \m(\Omega)(h_N(\Omega)+1)$, and $\m(\mathcal{E}_k(i)\bigtriangleup \mathcal{E}(i))\to0^+$ as $k\to+\infty$.  
Now, using \ref{prop:isoperim}, for all $k\in\mathbb{N}$ sufficiently large and any $i\in \{1,\ldots, N\}$, we get
\begin{equation}\label{eq:gerbillo}
    f(\m(\mathcal{E}_k(i)))\leq \frac{P(\mathcal{E}_k(i))}{\m(\mathcal{E}_k(i))}.
\end{equation}
The behavior of $f$ near zero prescribed by \ref{prop:isoperim} immediately implies that $\m(\mathcal{E}(i))\neq 0$ for all $i\in \{1,\ldots, N\}$, as otherwise a contradiction would arise with $h_N(\Omega)<+\infty$. Indeed, on the one hand, being $\{\mathcal{E}_k(i)\}_k$ a minimizing sequence, and owing to~\eqref{eq:gerbillo}, there exists $\bar k\gg1$ such that for all $k\ge \bar k$, we have
\[
f(\m(\mathcal{E}_k(i)))\leq 2h_N(\Omega).
\]
On the other hand, the isoperimetric property \ref{prop:isoperim} implies there exists $\delta>0$ such that $f(x)>2h_N(\Omega)$ for all $x\le \delta$. Hence, we deduce that $\m(\mathcal{E}_k(i)) \ge \delta$ for all $i=1,\dots, N$ and all $k\ge \bar k$.

It remains to be proved that $\mathcal E=\{\mathcal E(i)\}_{i=1}^N$ is an $N$-cluster contained in $\Omega$, i.e., that the chambers $\mathcal{E}(i)$ are pairwise disjoint, and the reader can easily check it on its own.

Consequently, thanks to~\ref{prop:lsc}, we find that
\begin{align*}
    h_N(\Omega)\leq \sum_{i=1}^N\frac{P(\mathcal{E}(i))}{\m(\mathcal{E}(i))}\leq \sum_{i=1}^N\liminf_{k\to+\infty}\frac{P(\mathcal{E}_k(i))}{\m(\mathcal{E}_k(i))}\leq \liminf_{k\to+\infty} \sum_{i=1}^N\frac{P(\mathcal{E}_k(i))}{\m(\mathcal{E}_k(i))}=h_N(\Omega),
\end{align*}
and the conclusion follows.
\end{proof}

Let us point out that properties \ref{prop:lsc}, \ref{prop:compact}, and~\ref{prop:isoperim} are all crucial in the above proof. Among them \ref{prop:isoperim} looks as the ``most artificial''; nevertheless, it is essential in the reasoning:  an example where existence fails when \ref{prop:isoperim} is missing is given in \cref{ex:non_existence_without_P6} below. It is also relevant to point out that these properties provide a sufficient but in no way a necessary condition, as \cref{ex:existence_without_P5_P6} and \cref{ex:existence_without_P5} show.

\begin{example}\label{ex:non_existence_without_P6}
Consider the measure space $(X,\mathscr{A},\m)=(\mathbb{R}^2,\mathscr{B}(\mathbb{R}^2), w\mathscr{L}^2)$, where $\mathscr{B}(\mathbb{R}^2)$ denotes the Borel $\sigma$-algebra, $w\in L^1(\mathbb{R}^2)$ is defined by
\[
w(x)=\begin{cases}
\|x\|^{-\frac32}, &\text{if } \|x\|\leq 1,\\[1ex]
e^{-\|x\|},&\text{if } \|x\|> 1,
\end{cases}
\]
and $P(\cdot)$ is the Euclidean perimeter. In this setting, properties \ref{prop:empty} through \ref{prop:compact} hold, but \ref{prop:isoperim} does not.

Within this framework, one has $h_1(\Omega)=0$, for any set $\Omega$ containing an open neighborhood of the origin. Indeed, it is enough to consider the sequence of balls centered at the origin $B_r \subset \Omega$ (for $r$ sufficiently small), for which we have
\[
P(B_r) = 2\pi r,  
\]
and
\[
\m(B_r) = \int_{B_r} w(x)\, \de x = 2\pi \int_0^r \rho^{-\frac32} \rho\, \de \rho = 4 \pi r^{\frac 12}.
\]
Were to exist $E\in \C_1(\Omega)$, then $P(E)=0$, and by the Euclidean isoperimetric inequality we would have $|E|=0$. Being the weight $w\in L^1(\R^2)$, this would eventually lead to 
\[
\m(E)=\int_{E} w(x)\, \de x=0,
\]
contradicting the fact that $\m(E)>0$. This shows that Cheeger sets do not exist. In more generality, the same happens in any measure space $(X, \A, \m)$ and in any $1$-admissible set $\Omega\in \A$ such that $h_1(\Omega)=0$ and the only measurable subsets $E$ of $\Omega$ with $P(E)=0$ have zero $\m$-measure.

For the sake of completeness, we shall note that, in the situation depicted in this remark, $N$-Cheeger sets exist in any open set $\Omega$ not containing the origin, since the weight $w$ would be $L^\infty(\Omega)$, refer to~\cite{BM16}*{Prop.~3.3} or to~\cite{Saracco18}*{Prop.~3.2}.
\end{example}

We now present two simple examples in which the existence of Cheeger sets is ensured even if properties~\ref{prop:compact} and~\ref{prop:isoperim} do not hold.

\begin{example}\label{ex:existence_without_P5_P6}
Consider any non-negative ($\sigma$-finite) measure space $(X,\A, \m)$, and consider  $P(E)=\m(E)$, for all $E\in \A$, as perimeter functional. For this choice, while~\ref{prop:lsc} holds, neither property~\ref{prop:compact} nor~\ref{prop:isoperim} hold, the latter because any isoperimetric function $f$ is bounded from above by $1$.  Nevertheless, fixed any $\Omega \in \A$, we have $h_N(\Omega) = N$, for any integer $N$, and any $N$-cluster is an $N$-Cheeger set.
\end{example}

\begin{example}\label{ex:existence_without_P5}
Consider any non-negative ($\sigma$-finite) measure space $(X,\A, \m)$, and consider $P(E)=0$, for all $E\in \A$, as perimeter functional. While~\ref{prop:lsc} holds, neither property~\ref{prop:compact} nor~\ref{prop:isoperim} hold. Nevertheless, fixed any $\Omega \in \A$, we have $h_N(\Omega) = 0$, for any integer $N$, and any $N$-cluster is an $N$-Cheeger set.
\end{example}

\subsection{Inequalities between the \texorpdfstring{$N$}{N}- and \texorpdfstring{$M$}{M}-Cheeger constants}\label{sec:ineq_M_N_constants}

\begin{proposition}\label{prop:ineq_M_N_constants}
Let $\Omega\in\A$ be an $N$-admissible set. Then, for all $M\in \N$ with $M<N$, one has
\begin{equation}\label{eq:general_M&N-Cheeger_constants}
h_M(\Omega) + h_{N-M}(\Omega) \le h_N(\Omega).
\end{equation}
\end{proposition}

\begin{proof}
Let $M$ and $N$ be fixed integers, with $M<N$. Let $\mathcal{E}$ be any fixed $N$-cluster. For any subset $J_M$ of $\{1,\dots, N\}$ of cardinality $M$, the $M$-cluster $\{\mathcal{E}(i)\}_{i\in J_M}$ provides an upper bound to $h_M(\Omega)$, whereas the $(N-M)$-cluster $\{\mathcal{E}(i)\}_{i\notin J_M}$ to $h_{N-M}(\Omega)$.

Hence, no matter how we choose $J_M$, we have
\[
\sum_{i=1}^N \frac{P(\mathcal{E}(i))}{\m(\mathcal{E}(i))} = \sum_{i\in J_M} \frac{P(\mathcal{E}(i))}{\m(\mathcal{E}(i))} + \sum_{i\notin J_M} \frac{P(\mathcal{E}(i))}{\m(\mathcal{E}(i))} \ge h_M(\Omega) + h_{N-M}(\Omega).
\]
By taking the infimum among all $N$-clusters, the desired inequality follows.
\end{proof}

\begin{corollary}
Let $\Omega\in\A$ be an $N$-admissible set. Then, for all $M\in \N$ such that for some integer $k$ one has $N=kM$, one has
\begin{equation}\label{eq:M&N-Cheeger_constants}
kh_M(\Omega) \le h_N(\Omega).
\end{equation}
\end{corollary}

\begin{remark}\label{rem:(N+2)-barbells}
The inequalities~\eqref{eq:general_M&N-Cheeger_constants} and~\eqref{eq:M&N-Cheeger_constants} hold as equalities in some cases, as, for instance, it happens anytime a set has multiple disjoint $1$-Cheeger sets. A trivial example of this behavior is given by $N$ disjoint and equal balls in the usual Euclidean space.

One can also build connected sets that have this feature. For $N=2$, it is enough to consider a standard dumbbell in the usual $2$-dimensional Euclidean space, that is, the set given by two disjoint equal balls, spaced sufficiently far apart, and connected via a thin tube. Such a set has two connected $1$-Cheeger sets $\mathcal{E}(1)$ and $\mathcal{E}(2)$ given by small perturbations of the two balls, and the $2$-cluster $\mathcal{E}=\{\mathcal{E}(i)\}$ is necessarily a $2$-Cheeger set, refer, for instance, to \cite{LP16}*{Ex.~4.5}.

An easy connected example for $N>2$ is instead given by an $(N+2)$-dumbbell in the usual $2$-dimensional Euclidean space, that is, a set formed by $N+2$ disjoint equal balls and linked by a thin tube, say
\[
\bigcup_{i=1}^{N+2} B_1((4i, 0)) \cup \Big((4, 4(N+2))\times(-\varepsilon, +\varepsilon)\Big),
\]
where $B_1((4i, 0))$ denotes the $2$-dimensional Euclidean ball of radius $1$ centered at the point $(4i,0)\in\R^2$. 
For $\varepsilon$ sufficiently small, and arguing as in \cite{LP16}*{Ex.~4.5}, it can be shown that such a set has $N$ connected and disjoint $1$-Cheeger sets, each corresponding to a small perturbation of the $N$ balls with two neighboring ones.
\end{remark}

\subsection{\texorpdfstring{$M$}{M}-subclusters of \texorpdfstring{$N$}{N}-Cheeger sets}
Given an $N$-Cheeger set of $\Omega$, consider any of its $M$-subcluster. It is natural to imagine that such an $M$-cluster is an $M$-Cheeger set in the ambient space given by $\Omega$ minus the $N-M$ chambers not belonging to the subcluster. In this short section, we prove that this is true.

For the sake of clarity of notation, we let $|J|\in\N\cup\set{0}\cup\set{+\infty}$ be the cardinality of a set $J\subset\N$.

\begin{proposition}\label{lem:M-Cheeger_from_N-Cheeger}
Let $\Omega\in\mathscr{A}$ be an $N$-admissible set, and assume that it has an $N$-Cheeger set $\mathcal{E} = \{\mathcal{E}(i)\}_{i=1}^N \in \C_N(\Omega)$.
For any proper subset $J\subset \{1,\dots, N\}$, let
\begin{equation}\label{def:Omega_J}
\Omega_J = \Omega \setminus \bigcup_{j\notin J} \mathcal{E}(j),
\end{equation}
and let $\mathcal{E}_J$ be the $|J|$-cluster given by
\[
\mathcal{E}_J = \{\mathcal{E}(j)\}_{j\in J}.
\]
Then, $\mathcal{E}_J$ is a $|J|$-Cheeger set of $\Omega_J$.
\end{proposition}

\begin{proof}
It is enough to prove the claim for a subset $J$ of cardinality $N-1$, and then to reason by induction. In particular, up to relabeling, we can assume $J$ to be the proper subset $\{1,\dots, N-1\}$.

As both $\Omega$ and $\mathcal{E}(N)$ are measurable, so it is the set $\Omega_{J}$. Moreover, this latter is $(N-1)$-admissible because there exists at least the $(N-1)$-cluster $\{\mathcal{E}(i)\}_{i=1}^{N-1}$.

By contradiction, assume that $\{\mathcal{E}(i)\}_{i=1}^{N-1}$ is not an $(N-1)$-Cheeger set of $\Omega_J$. Then, for $\varepsilon$ small enough, we find a different $(N-1)$-cluster $\{\mathcal{F}(i)\}_{i=1}^{N-1}$ with
\[
\sum_{i=1}^{N-1}\frac{P(\mathcal{F}(i))}{\m(\mathcal{F}(i))} < h_{N-1}(\Omega)+\varepsilon < \sum_{i=1}^{N-1}\frac{P(\mathcal{E}(i))}{\m(\mathcal{E}(i))}.
\]
It is then immediate that the $N$-cluster
\[
\{\mathcal{F}(i)\}_{i=1}^{N} = \{\mathcal{F}(1), \dots, \mathcal{F}(N-1), \mathcal{E}(N) \}
\]
contradicts the minimality of the $N$-cluster $\{\mathcal{E}(i)\}_{i=1}^{N}$ in $\Omega$.
\end{proof}

\subsection{Properties of \texorpdfstring{$N$}{N}-Cheeger sets}

\begin{proposition}
[Basic properties of $N$-Cheeger sets]
Let $\{\Omega_k\}_k\subset\A$ be a collection of $N$-admissible sets. The following hold for all integers $M\le N$:

\begin{enumerate}[label=(\roman*)]
    \item\label{item:beatles} if $\Omega_1\subset \Omega_2$, then $h_M(\Omega_1)\ge h_M(\Omega_2)$;
    \item\label{item:abba} if~\ref{prop:isoperim} is in force, and $\m(\Omega_k)\to0^+$, then $h_M(\Omega_k)\to+\infty$;
    \item\label{item:queen} if~\ref{prop:lsc}, \ref{prop:compact}, and~\ref{prop:isoperim} are in force, and $\Omega_k\to\Omega$ in $L^1(X, \m)$, with $\m(\Omega)\in (0,+\infty)$, then
    \[
    h_M(\Omega)\le\liminf_k h_M(\Omega_k).
    \]
    Moreover, if also~\ref{prop:sub-mod} is in force, $P(\Omega)$ is finite, and $P(\Omega_k) \to P(\Omega)$, then
    \[
    h_M(\Omega)= \lim_k h_M(\Omega_k).
    \]
\end{enumerate}
\end{proposition}

\begin{proof}
Recall that an $N$-admissible set $\Omega$ is also $M$-admissible for all integers $M\le N$, see \cref{rem:Nadm=>Madm}.

\vspace{1ex}

\textit{Proof of \ref{item:beatles}}.
For any two fixed $N$-admissible sets with $\Omega_1\subset \Omega_2$, any $M$-cluster of $\Omega_1$ is also an $M$-cluster of $\Omega_2$. The inequality immediately follows by definition of $M$-Cheeger constant.

\vspace{1ex}

\textit{Proof of~\ref{item:abba}}.
In virtue of \eqref{eq:M&N-Cheeger_constants} and the positivity of $h_M(\Omega)$, it is enough to prove the claim for $M=1$. Fix $\varepsilon>0$, and for all $k$, let $C_k \subset \Omega_k$ be such that
\[
h_1(\Omega_k) + \varepsilon \ge \frac{P(C_k)}{\m(C_k)}.
\]
Then, by~\ref{prop:isoperim}, we have
\[
h_1(\Omega_k) +\varepsilon \ge f(\m(C_k)),
\]
and the claim follows by the monotonicity of the measure paired with the hypothesis that the $\m$-measure of $\Omega_k$ vanishes, that is, $\m(C_k)\le \m(\Omega_k) \to 0$, and the behavior of $f$ prescribed by~\ref{prop:isoperim}.

\vspace{1ex}

\textit{Proof of \ref{item:queen}}. 
Without loss of generality, we can assume there exists a constant $C_1<+\infty$ independent of $k$ such that
\[
\liminf_k h_M(\Omega_k)\le C_1,
\]
as otherwise there is nothing to prove. Let us consider the (not relabeled) sequence realizing the $\liminf$. Since $\Omega_k$ is converging in $L^1(X, \m)$ to a set of finite $\m$-measure, we can also assume that $\m(\Omega_k)$ is equibounded, independently of $k$, that is,
\[
\m(\Omega_k)\le C_2.
\]
Thus, by \cref{res:existence_N-Cheeger}, for each $k$, there exists an $M$-Cheeger set $\mathcal{E}_k$ for $\Omega_k$. Moreover, for any $k$, we have
\begin{equation}\label{eq:key_ineq_liminf_hM}
\sum_{i=1}^M P(\mathcal{E}_k(i)) \le h_M(\Omega_k) \m(\Omega_k) \le C.
\end{equation}
Hence, by~\ref{prop:compact}, we can extract a (not relabeled) subsequence $\{\mathcal{E}_k(i)\}_k$ such that for all indexes $i$, the chamber $\mathcal{E}_k(i)$ converges in $L^1(X, \m)$ to a limit set $\mathcal{E}(i)$ necessarily contained in $\Omega$ up to $\m$-null sets. Moreover, by~\ref{prop:isoperim}, one necessarily has $\m(\mathcal{E}(i))>0$ as otherwise a contradiction with the finiteness of $\liminf_k h_M(\Omega_k)$ would arise. Hence, $\mathcal{E}$ is an $M$-cluster of $\Omega$. Thus, owing to~\ref{prop:lsc} and the fact that $\mathcal{E}_k$ is an $M$-Cheeger cluster of $\Omega_k$, we have
\[
h_M(\Omega) \le \sum_{i=1}^M \frac{P(\mathcal{E}(i))}{\m(\mathcal{E}(i))} \le \liminf_k \sum_{i=1}^M \frac{P(\mathcal{E}_k(i))}{\m(\mathcal{E}_k(i))} = \liminf_k h_M(\Omega_k),
\]
that is, the first part of the claim.

To show the second part, let us pick an $M$-Cheeger cluster $\mathcal{E}$ of $\Omega$, which exists since we are under the assumptions of \cref{res:existence_N-Cheeger}. Let us consider the collections
\[
\{\mathcal{E}_k(i) = \mathcal{E}(i) \cap \Omega_k\},
\]
which are $M$-clusters of $\Omega_k$ for $k$ sufficiently large.
Clearly, for each fixed $i$, we have that $\mathcal{E}_k(i)$ converges in $L^1(X, \m)$ to $\mathcal{E}(i)$, while $\mathcal{E}_k(i) \cup \Omega_k$ to $\Omega$. Therefore, by~\ref{prop:sub-mod}, for each $i$, we have
\begin{align*}
    P(\mathcal{E}_k(i)) \le P(\mathcal{E}(i)) + P(\Omega_k) - P(\mathcal{E}(i)\cup \Omega_k).
\end{align*}
Taking the $\limsup_k$, using the assumption of the convergence of $P(\Omega_k)$, we have
\[
\limsup_k P(\mathcal{E}_k(i)) \le  P(\mathcal{E}(i)) + P(\Omega) - \liminf_k P(\mathcal{E}(i)\cup \Omega_k) \le P(\mathcal{E}(i)). 
\]
Together with~\ref{prop:lsc} this implies that, for each $i$, $\lim_k P(\mathcal{E}_k(i))$ exists and equals $P(\mathcal{E}(i))$. Combining this fact with the minimality of $\mathcal{E}$ for $h_M(\Omega)$ and the first part of the claim, we conclude the proof.
\end{proof}

\begin{remark}
Notice that to prove point~\ref{item:queen}, all requests come into play. Indeed, a priori, one can work with an ``almost-infimizing'' $M$-cluster for $h_M(\Omega_k)$ and find an analogous of~\eqref{eq:key_ineq_liminf_hM} up to an additive factor $\varepsilon\m(\Omega_k)$. Then, the compactness granted by~\ref{prop:compact} is needed, but in order to ensure that the limiting collection $\mathcal{E}$ is indeed a cluster, the isoperimetric property~\ref{prop:isoperim} is needed. Finally, when talking about a $\liminf$ property, we cannot avoid enforcing~\ref{prop:lsc}. 
\end{remark}

\begin{lemma}\label{lem:lower_bound_volume}
Let $\Omega\in\A$ be an $N$-admissible set with $\m(\Omega)\in (0,+\infty)$, and assume that $\C_N(\Omega)\ne\emptyset$. If~\ref{prop:isoperim} is in force, then for every $N$-Cheeger set $\{\mathcal{C}(i)\}_{i=1}^N\in \C_N(\Omega)$, and for every $i=1,\dots, N$, we have
\[
\m(\mathcal{C}(i)) \ge c,
\]
where $c>0$ is a constant depending only on $h_N(\Omega)$ and the isoperimetric function $f$ appearing in~\ref{prop:isoperim}. 
\end{lemma}

\begin{proof}
Let $\{\mathcal{C}(i)\}_{i=1}^N\in \C_N(\Omega)$, then owing to~\ref{prop:isoperim}, for every $i$ we have
\[
+\infty> h_N(\Omega) \ge \frac{P(\mathcal{C}(i))}{\m(\mathcal{C}(i))} \ge f(\m(\mathcal{C}(i))).
\]
By the assumptions on $f$ given in~\ref{prop:isoperim}, $f(\varepsilon)\to +\infty$ as $\varepsilon\to 0^+$. Thus, there exists a threshold $c=c(h_N(\Omega), f)>0$ such that 
\[
f(\m(E)) > h_N(\Omega) \quad \text{for all}\ E\in \A, \text{ with } \m(E)<c.
\]
Hence, for any $N$-Cheeger set, the lower bound on the volume of each of its chambers follows.
\end{proof}

\subsection{Additional properties of \texorpdfstring{$1$}{1}-Cheeger sets}

For $1$-Cheeger sets, something more can be said in general, as we show in the next proposition.

\begin{proposition}\label{prop:closure_wrt_unions_intersections}
Let $\Omega\in\A$ be a $1$-admissible set, and assume that $\C_1(\Omega)$ is not empty. If~\ref{prop:sub-mod} is in force, for any $E, F \in \C_1(\Omega)$, the following hold:
\begin{enumerate}[label=(\roman*)]
\item\label{item:aglio} 
 $E\cup F\in\C_1(\Omega)$;
\item\label{item:cipolla}
$E\cap F\in\C_1(\Omega)$, provided that $\m(E\cap F)>0$.
\end{enumerate}
Moreover, if also~\ref{prop:lsc} is in force, then $\C_1(\Omega)$ is closed with respect to countable unions and $\m$-non-negligible intersections, that is, given any countable family $\{E_j\}_j$ of $1$-Cheeger sets, one has:
\begin{enumerate}[label=(\roman*), resume]
\item\label{item:carota}
$\bigcup_{j}E_j\in\C_1(\Omega)$;
\item\label{item:sedano}
$\bigcap_{j}E_j\in\C_1(\Omega)$, provided that $\m\left(\bigcap_{j}E_j\right)>0$.
\end{enumerate}
\end{proposition}

\begin{proof}
First, notice that we have the equalities
\[
P(E)= h_1(\Omega)\m(E),
\qquad
P(F)= h_1(\Omega)\m(F).
\]
Hence, owing to~\ref{prop:sub-mod}, the following chain of inequalities holds:
\begin{align*}
h_1(\Omega)(\m(E\cup F) + \m(E\cap F)) 
&= h_1(\Omega)(\m(E)+\m(F))\\
&= P(E) + P(F)\\
&\ge P(E\cup F) + P(E\cap F)\\
&\ge h_1(\Omega) (\m(E\cup F) + \m(E\cap F)),
\end{align*}
and thus, they are all equalities. This implies that
\begin{align*}
P(E\cap F) &= h_1(\Omega)\m(E\cap F),\\
P(E\cup F) &= h_1(\Omega)\m(E\cup F),
\end{align*}
thus $E\cup F$ is a $1$-Cheeger set, and so it is $E\cap F$, provided that it has positive $\m$-measure. This settles points~\ref{item:aglio} and~\ref{item:cipolla}.

Now, concerning the proof of~\ref{item:carota} and~\ref{item:sedano}, let $\{E_j\}_j$ be any countable family of $1$-Cheeger sets. Then, the sequences
\[
F_k=\bigcup_{j=1}^k E_j,\qquad G_k=\bigcap_{j=1}^k E_j,
\]
are sequences of $1$-Cheeger sets by points~\ref{item:aglio} and~\ref{item:cipolla} previously established (the second one, under the additional assumption that the intersections are $\m$-non-negligible). Moreover, they converge, respectively, in $L^1(X,\m)$ to the sets
\[
F = \bigcup_{j} E_j,\qquad G=\bigcap_{j} E_j.
\]
The lower-semicontinuity of $P$ granted by~\ref{prop:lsc} implies that these sets are $1$-Cheeger sets themselves, and this concludes the proof.
\end{proof}

\begin{proposition}[Maximal Cheeger set]\label{prop:maximal_sets}
Let $\Omega\in\A$ be a $1$-admissible set, with $\m(\Omega)\in (0,+\infty)$, and assume that $\C_1(\Omega)$ is not empty. If~\ref{prop:lsc} and~\ref{prop:compact} are in force, then there exist $1$-Cheeger sets with maximal measure, and we shall call them \emph{maximal $1$-Cheeger sets}.
If also~\ref{prop:sub-mod} holds, then there exists a unique (up to $\m$-negligible sets) maximal $1$-Cheeger set $E^+\in\C_1(\Omega)$, and it is such that $E\subset E^+$ for all $E\in\C_1(\Omega)$.
\end{proposition}

\begin{proof}
Take a sequence $\{C_k\}_k$ in $\C_1(\Omega)$ supremizing the $L^1(X,\m)$ norm. The following uniform upper bound on the perimeters of $\{C_k\}_k$ holds
\[
P(C_k) = h_1(\Omega)\m(C_k) \le h_1(\Omega)\m(\Omega).
\]
Thus, by~\ref{prop:compact}, up to extracting a subsequence, $C_k$ converge to some limit set $C$, which, by~\ref{prop:lsc}, is readily proven to be a $1$-Cheeger set itself, provided that $\m(C)>0$, which holds true as we look for sets maximizing the $L^1(X,\m)$ norm.

Now additionally assume~\ref{prop:sub-mod}, and let $C_0$ be a maximal $1$-Cheeger set. For any other $1$-Cheeger set $C$, if one were not to have $C\subset C_0$ a contradiction would immediately ensue, since $C\cup C_0$ is itself a $1$-Cheeger set by \cref{prop:closure_wrt_unions_intersections}\ref{item:aglio}. This same reasoning also yields the uniqueness of such a set.
\end{proof}

\begin{example}\label{ex:no_maximal}
An example of metric-measure space $(X, \A, \m)$ where existence of maximal $1$-Cheeger sets fails is the following one. Consider $(X, \A, \m)$ as the probability measure space $(\R^2, \mathscr{B}(\R^2), (2\pi)^{-\frac 12}e^{-\|\,\cdot\,\|}\mathscr{L}^2)$, and consider the perimeter functional given by $P(E) = 0$ for all $E \in \A$ excluded $\R^2$ itself, and $P(\R^2) = 1$. In such a setting neither~\ref{prop:lsc} nor~\ref{prop:compact} hold. If we choose $\Omega =\R^2$, all $\m$-non-negligible sets are $1$-Cheeger sets but for the whole space $\R^2$. No maximal $1$-Cheeger set exists, as the supremum of their norms is $1$, and this is the measure of the lone $\R^2$.
\end{example}

\begin{proposition}[Minimal Cheeger set]\label{prop:minimal_sets}
Let $\Omega\in\A$ be a $1$-admissible set with finite $\m$-measure, and assume that $\C_1(\Omega)$ is not empty. If~\ref{prop:lsc}, \ref{prop:compact}, and~\ref{prop:isoperim} are in force, then there exist $1$-Cheeger sets with minimal $L^1(X, \m)$ norm, and we shall call them \emph{minimal $1$-Cheeger sets}.
\end{proposition}

\begin{proof}
The proof is exactly the same as the one of \cref{prop:maximal_sets}, except that we now need to ensure that the limit set $C$ is a viable competitor, that is, $\m(C)>0$. This is exactly why we need to require~\ref{prop:isoperim}. The uniform lower bound on the volume of any $1$-Cheeger set provided by \cref{lem:lower_bound_volume} immediately allows to conclude.
\end{proof}

\begin{example}
An example of measure space $(X, \A, \m)$ where existence of minimal $1$-Cheeger sets fails is the one of \cref{ex:no_maximal}. Chosen any $\Omega$, all of its subsets but $\m$-negligible ones are $1$-Cheeger sets. Hence, minimal $1$-Cheeger sets do not exist, being the infimum of the measures of $1$-Cheeger sets equal to zero.
\end{example}

\section{Sets with prescribed mean curvature}
\label{sec:PMC}

In this section, we work in a measure space endowed with a perimeter functional as in \cref{ssec:perimeter_functional}. We show that the Cheeger constant acts as a threshold to determine whether non-trivial minimizers exist for the so-called \emph{prescribed mean curvature functional}, under some suitable assumptions on the perimeter functional. In order to do so, the first result we need to prove is the following lemma.

\begin{lemma}\label{lem:PMC}
Let $\Omega\in\A$ be a $1$-admissible set. An $\m$-non-negligible set $C$ is a $1$-Cheeger set of $\Omega$ if and only if it is a minimizer of
\begin{equation}\label{eq:PMC_h}
\mathcal{J}_{h_1(\Omega)}[F]=P(F) - h_1(\Omega) \m(F)	
\end{equation}
among $\set*{F\in \mathscr{A}\,:\,F\subset \Omega}$.
\end{lemma}

\begin{proof}
It is sufficient to note that the inequality $\mathcal{J}_{h_1(\Omega)}[F] \ge 0$ holds true, and that the only $\m$-non-negligible sets that can saturate it are $1$-Cheeger sets. The non-negativity of $\mathcal{J}_{h_1(\Omega)}$ is trivial for $\m$-negligible sets, while for all $\m$-non-negligible sets, it follows by the definition of $1$-Cheeger constant.
\end{proof}

\subsection{The \texorpdfstring{$P$}{P}-mean curvature}\label{ssec:vmc}

\cref{lem:PMC} allows us to infer that any $1$-Cheeger set has a \emph{$P$-mean curvature}, defined as follows.

\begin{definition}\label{defin:P_mean_curvature}
Let $(X,\A, \m)$ be a measure space, and let $\Omega\in\A$. A function $H\in L^1(\Omega, \m)$ is said to be a \emph{$P$-mean curvature} in $\Omega$ of a set $E\subset \Omega$ if $E$ minimizes the functional
\begin{equation*}
F\mapsto
P(F) -\int_F H\, \de\m,
\end{equation*}
among all $\m$-measurable $F\subset\Omega$.
\end{definition}

A similar definition was first given in measure spaces in~\cite{BM16} under some assumptions on the perimeter functional, and only when $\Omega=X$, with the additional request that $\m(X)<+\infty$. 

Given \cref{defin:P_mean_curvature}, a direct consequence of \cref{lem:PMC} is that $1$-Cheeger sets of $\Omega$, if they exist, have $h_1(\Omega)$ as one of their $P$-mean curvatures in $\Omega$---and this without assuming anything on the perimeter, apart from the non-negativity.

\begin{corollary}\label{cor:variational_curvature_Cheeger}
Let $(X, \A, \m)$ be a non-negative $\sigma$-finite measure space and let $\Omega\in \A$ be a $1$-admissible set. If $C\in \C_1(\Omega)$ is a $1$-Cheeger set in $\Omega$, then $h_1(\Omega)$ is a  $P$-mean curvature in $\Omega$ of $C$.
\end{corollary}

A similar result can be inferred on each chamber of an $N$-Cheeger cluster, by simply using \cref{lem:M-Cheeger_from_N-Cheeger}.

\begin{corollary}\label{cor:variational_curvature_N-Cheeger}
Let $(X, \A, \m)$ be a non-negative $\sigma$-finite measure space and let $\Omega\in \A$ be an $N$-admissible set. If $\mathcal{E} = \{\mathcal{E}(i)\}_{i=1}^N \in \C_N(\Omega)$, then, letting $J_i=\{i\}$, for every $i=1,\dots, N$, $h_1(\Omega_{J_i})$ is a $P$-mean curvature in $\Omega_{J_i}$ of the chamber $\mathcal{E}(i)$, where $\Omega_{J_i}$ is as in~\eqref{def:Omega_J}. 
\end{corollary}

\begin{proof}
Let $i$ be fixed. 
By \cref{lem:M-Cheeger_from_N-Cheeger}, the chamber $\mathcal{E}(i)$ is a $1$-Cheeger set of $\Omega_{J_i}$. The conclusion now directly follows from \cref{cor:variational_curvature_Cheeger}.
\end{proof}

\subsection{Relation with sets with prescribed \texorpdfstring{$P$}{P}-mean curvature}

Let now $\Omega \in \A$ be fixed, and consider the functional
\begin{equation}\label{eq:PMC_k}
\mathcal{J}_\kappa[F]=P(F) - \kappa \m(F),
\end{equation}
that is, the same functional introduced in~\eqref{eq:PMC_h} but with a general positive constant $\kappa \in \R_+$ in place of $h_1(\Omega)$. The reason why this functional is referred to as the prescribed $P$-mean curvature functional, is that $\kappa$ is a $P$-mean curvature of any minimizer $E_\kappa$ of the functional $\mathcal J_\kappa$ in~\eqref{eq:PMC_k}.

The next theorem states that if there exists a $1$-Cheeger set, and properties~\ref{prop:empty}, \ref{prop:lsc}, and~\ref{prop:compact} are in force, then~\eqref{eq:PMC_k} has $\m$-non-negligible minimizers if and only $\kappa\ge h_1(\Omega)$.

\begin{theorem}\label{prop:min_PMC}
Let $(X,\A, \m)$ be a non-negative $\sigma$-finite measure space, and let $\Omega\in\A$ with finite $\m$-measure. For $\kappa>0$, let $\mathcal{J}_\kappa$  be the functional
\[
\mathcal{J}_\kappa[F]=P(F) - \kappa \m(F),
\]
defined over $\m$-measurable subsets of $\Omega$.

Then, if properties~\ref{prop:lsc} and~\ref{prop:compact} are in force, minimizers of $\mathcal J_\kappa$ exist.
In addition, if property~\ref{prop:empty} is also in force, the following hold true:
\begin{enumerate}[label=(\roman*)]
    \item\label{item:nina} if $\mathcal{J}_\kappa$ has $\m$-non-negligible minima, then $\kappa \ge h_1(\Omega)$;
    \item\label{item:pinta} if $\kappa > h_1(\Omega)$, then $\mathcal{J}_\kappa$ has $\m$-non-negligible minima.
\end{enumerate}
Moreover, 
\begin{enumerate}[label=(\roman*),resume]
    \item\label{item:santa_maria} if $\Omega$ has a $1$-Cheeger set $C\in \C_1(\Omega)$, then $\mathcal{J}_\kappa$ has $\m$-non-negligible minima if and only if $\kappa \ge h_1(\Omega)$.
\end{enumerate}
\end{theorem}

\begin{proof}
First, we show that assumptions~\ref{prop:lsc}, \ref{prop:compact}, and the finiteness of $\m(\Omega)$ imply the existence of minimizers of $\mathcal J_\kappa$.

Indeed, since the perimeter functional is non-negative, and by the monotonicity of measures, we have the trivial lower bound $\mathcal{J}_\kappa \ge -\kappa\m(\Omega)$. Therefore, we can take an infimizing sequence $\{E_k\}_k$, for whose perimeters (for $k$ large enough) we have the uniform upper bound
\[
P(E_k) \le \kappa \m(\Omega) + \inf \mathcal{J}_\kappa + 1,
\]
which is finite by the finiteness of $\m(\Omega)$.

By~\ref{prop:compact}, we can extract a converging subsequence in $L^1(X,\m)$ to some limit set $E$, and by~\ref{prop:lsc}, we have
\[
P(E) - \kappa\m(E) \le \liminf_k\Big( P(E_k) - \kappa \m(E_k)\Big) = \inf \mathcal{J}_\kappa.
\]
Hence, $E$ is a minimizer of the problem.

Let us now turn our attention to points~\ref{item:nina}, \ref{item:pinta} and~\ref{item:santa_maria}. First, note that requiring~\ref{prop:empty} implies that the minimum of $\mathcal{J}_\kappa$ is non-positive. Indeed, owing also to~\ref{prop:lsc}, \cref{rem:negligible_modifications} gives that $\mathcal{J}_\kappa$ is zero whenever evaluated on an $\m$-negligible set.

Suppose that there exists an $\m$-non-negligible minimizer $E_\kappa$. Necessarily, by comparing with an $\m$-negligible set we have
\[
P(E_\kappa) - \kappa \m(E_\kappa) \le 0,
\]
which, rearranged, gives $\kappa \ge P(E_\kappa)\m(E_\kappa)^{-1}$, and this ratio has to be greater than or equal to $h_1(\Omega)$ by definition; thus, point~\ref{item:nina} is settled.

Conversely, let $\kappa > h_1(\Omega)$, and let $\varepsilon>0$ be such that $\kappa = h_1(\Omega)+2\varepsilon$. Since $\Omega$ has finite $\m$-measure, just as in the proof of \cref{res:existence_N-Cheeger}, we can find an infimizing sequence $\{C_j\}_j$ for the $1$-Cheeger constant $h_1(\Omega)$. For $j\gg1$, we have $P(C_j)\m(C_j)^{-1} \le h_1(\Omega)+\varepsilon$. Hence,
\[
\min \mathcal{J}_\kappa \le P(C_j)-\kappa\m(C_j) < P(C_j)-(h_1(\Omega)+\varepsilon)\m(C_j) \le 0,
\]
which yields the claim since $\mathcal{J}_\kappa$ is zero when evaluated on $\m$-negligible sets. This establishes point~\ref{item:pinta}.

Finally, assuming the existence of a $1$-Cheeger set, point~\ref{item:santa_maria} follows directly from~\ref{item:nina}, \ref{item:pinta}, and \cref{lem:PMC}.
\end{proof}

\section{Relation with first \texorpdfstring{$1$}{1}-eigenvalue}

\label{sec:realtion_with_1-eigen}

In this section, we work in the setting of \cref{ssec:BV_measure_spaces}, where we introduced $BV$ functions starting from a given perimeter functional $P$ on a measure space $(X,\A,\m)$.  

\subsection{First \texorpdfstring{$1$}{1}-eigenvalue for \texorpdfstring{$N$}{N}-clusters}

For a given $\m$-measurable subset $\Omega\subset X$, we let
\begin{equation*}
BV_0(\Omega,\m)
= 
\set*{u\in BV(X, \m): u|_{X\setminus \Omega}=0}.	
\end{equation*}
Here and in the following, we write $u|_{X\setminus\Omega}=0$ whenever 
\begin{equation*}
\int_{X\setminus\Omega}|u|\,\de\m=0.    
\end{equation*}
Thanks to \cref{res:basic_props_var}\ref{item:prop_var_constant}, under the validity of  assumptions~\ref{prop:empty} and~\ref{prop:space}, we have $ BV_0(\Omega,\m)\ne\emptyset$.

\begin{definition}[First $1$-eigenvalue]
\label{def:first_1_eigen}
Let properties~\ref{prop:empty} and~\ref{prop:space} be in force.
Let $\Omega\in\mathscr{A}$ be a $1$-admissible set with $\m(\Omega)\in (0,+\infty)$.
We call
\begin{equation*}
\lambda_{1,1}(\Omega)
=
\inf\set*{
\frac{\Var(u)}{\|u\|_1}
:
u\in BV_0(\Omega,\m),\ \|u\|_1>0
}
\in[0,+\infty)
\end{equation*} 
the \emph{first $1$-eigenvalue} relative to the variation on $\Omega$.
\end{definition}

Analogously, we can define the first $1$-eigenvalue in the case of $N$-clusters as follows.

\begin{definition}[First $1$-eigenvalue for $N$-clusters]
\label{def:first_N_eigen}
Let properties~\ref{prop:empty} and \ref{prop:space} be in force.
Let $\Omega\in\mathscr{A}$ be an $N$-admissible set with $\m(\Omega)\in (0,+\infty)$.
We define the \emph{first $1$-eigenvalue for $N$-clusters} relative to the variation on $\Omega$ as the quantity
\begin{equation}\label{eq:def_lambda_N}
\Lambda_{N}(\Omega)
=
\inf \sum_{i=1}^N\frac{\Var(u_i)}{\|u_i\|_1},
\end{equation}
where the infimum is sought among the $N$-tuples $\{u_i\}_{i=1}^N$ such that
\begin{equation*}
u_i\in BV_0(\Omega, \m),\ \text{with}\  \|u_i\|_1>0\ \text{and}\ \supp(u_i)\cap \supp(u_j)=\emptyset,    
\end{equation*}
for all $i\neq j$, $i,j=1,\dots,N$.
\end{definition}

Clearly, we have that $\Lambda_1(\Omega)= \lambda_{1,1}(\Omega)$ for all $1$-admissible sets $\Omega\in\mathscr A$ with $\m(\Omega)\in(0,+\infty)$.
In \cref{res:relation_with_first_1_eigen} below, we prove the relation 
$
\lambda_{1,1}(\Omega)=h_1(\Omega),
$
also see~\cite{KF03}*{Cor.~6}, under the additional~\ref{prop:complement_set}, as a consequence of a more general inequality involving $h_N(\Omega)$ and $\Lambda_{N}(\Omega)$ in the spirit of~\cite{CL19}*{Thm.~3.3}.

\subsection{Relation with first \texorpdfstring{$1$}{1}-eigenvalue for \texorpdfstring{$N$}{N}-clusters}
\label{ssec:realtion_with_1-eigen}

We need the following preliminary result, which can be seen as a symmetric version of the coarea formula~\eqref{eq:def_var_by_coarea}.
Note that, in \cref{lem:coarea_mod}, we assume the validity of the symmetry property~\ref{prop:complement_set}. 

\begin{lemma}[Symmetric coarea formula]\label{lem:coarea_mod}
Let properties~\ref{prop:empty}, \ref{prop:space}, \ref{prop:lsc}, and~\ref{prop:complement_set} be in force. 
If $u\in BV (X, \m)$ and
\[
F^t =
\begin{cases}
\set{u>t}, \quad &\text{if $t\ge 0$},
\\[2mm]
\set{u< t}, \quad &\text{if $t< 0$},
\end{cases}
\]
then
\begin{equation}\label{eq:cavalieri}
\|u\|_1 = \int_\R \m(F^t)\,\de t,
\end{equation}
and 
\begin{equation}\label{eq:coarea_mod}
\Var(u) = \int_\R P(F^t)\,\de t.
\end{equation}
In particular, $\chi_{F_t} \in BV (X, \m)$ for a.e.\ $t\in\R$. 
In addition, if $u\in BV_0(\Omega, \m)$, then $F^t\subset \Omega$ for all $t$.
\end{lemma}

\begin{proof}
Equation~\eqref{eq:cavalieri} is a simple consequence of Cavalieri's principle, being
\begin{align*}
\|u\|_1 
&= 
\int_X |u|\, \de\m = \int_X (u^+ + u^-)\, \de\m
\\
&= 
\int_0^{+\infty} \m(\{u^+ > t\})\, \de t + \int_0^{+\infty} \m(\{u^- > t\})\, \de t 
\\
&= 
\int_\R \m(F^t)\, \de t,
\end{align*}
where $u^+$ and $u^-$ are the positive and negative parts of $u$, respectively.

To prove~\eqref{eq:coarea_mod}, we first observe that, by property~\ref{prop:complement_set}, 
\begin{align*}
\Var(u)
&=
\int_{\R} P(\set*{u>t})\, \de t
\\
&=
\int_{-\infty}^0 P(\set*{u>t})\,\de t
+
\int_0^{+\infty} P(\set*{u>t})\,\de t
\\
&=
\int_{-\infty}^0 P(\set*{u\le t})\,\de t
+
\int_0^{+\infty} P(\set*{u>t})\,\de t.
\end{align*}
Therefore, we just need to show that $P(\set*{u\le t})=P(\set*{u<t})$ for a.e.\ ${t<0}$. 
Since clearly $\set*{u\le t}=\set*{u<t}\cup\set*{u=t}$, thanks to property~\ref{prop:lsc} and \cref{rem:negligible_modifications}, it is enough to prove that $\m(\set*{u=t})=0$ for a.e.\ $t<0$.
This is an immediate consequence of the fact that the function $t\mapsto\m(\set*{u\le t})$ is monotone non-decreasing for $t<0$, so that the set of its  discontinuity points $\set*{t<0 : \m(\set*{u=t})>0}$  is at most countable.

Now, the fact that $\chi_{F^t}\in BV(X,\m)$ for a.e.\ $t\in\R$ directly follows from~\eqref{eq:coarea_mod} and~\eqref{eq:cavalieri}. 
Finally, if $t>0$ then $F^t = \{u>t\}\subset \Omega$ by definition of $BV_0(\Omega, \m)$. 
In a similar way, if $t<0$, then
\[
F^t =\{u<t\} = X\setminus \{u\ge t\} \subset X\setminus \{u\ge 0\} \subset \Omega.
\]
The proof is complete.
\end{proof}

\begin{theorem}[Relation with first $1$-eigenvalue for $N$-clusters]
\label{res:relation_with_first_1_eigen}
Let properties \ref{prop:empty} and \ref{prop:space} be in force.
If $\Omega\in\mathscr{A}$ is an $N$-admissible set with $\m(\Omega)\in (0,+\infty)$, then
\[
\Lambda_N(\Omega) \le h_N(\Omega).
\]
Moreover, if also properties~\ref{prop:lsc} and~\ref{prop:complement_set} hold, then
\begin{equation*}
Nh_1(\Omega) \le \Lambda_N(\Omega),
\end{equation*}
and thus, in particular, $h_1(\Omega)=\lambda_{1,1}(\Omega)$.
\end{theorem}

\begin{proof}
Thanks to \cref{res:tot_var_of_sets}, as $\Omega$ is $N$-admissible, given any $N$-cluster $\mathcal E=\set*{\mathcal E(i)}_{i=1}^N$ in $\Omega$, the $N$-tuple $\{\chi_{\mathcal E(i)}:i=1,\dots,N\}$ is a viable competitor in the definition of $\Lambda_N(\Omega)$.
Hence, the inequality
\begin{equation*}
\Lambda_N(\Omega) \le h_N(\Omega)
\end{equation*}
immediately follows.

We now assume the validity of properties~\ref{prop:lsc} and~\ref{prop:complement_set}, so that we can use \cref{lem:coarea_mod}, and focus on the lower bound on $\Lambda_N(\Omega)$.

We begin with the case $N=1$.
By contradiction, start by assuming that $\Lambda_1(\Omega) < h_1(\Omega)$. 
Fixed any $\varepsilon>0$ such that $\Lambda_1(\Omega) +\varepsilon \le h_1(\Omega)$, we let $u\in BV_0(\Omega,\m)$ be a competitor in the definition of $\Lambda_1(\Omega)$ such that
\begin{equation}\label{eq:scoiattolo}
\frac{\Var(u)}{\|u\|_1} \le \Lambda_1(\Omega) +\varepsilon.
\end{equation}
Let $\set*{F^t : t\in\R}$ be the family of sets introduced in \cref{lem:coarea_mod} relatively to the function~$u$.
We claim that there exists $\bar t \in[0,+\infty)$ such that
\begin{equation}
\label{eq:armadillo}    
\text{$\m(F^t)>0$ either for all $t>\bar t$ or for all $t<-\bar t$.
}
\end{equation}
Indeed, if either $t_1<t_2<0$ or $t_1>t_2>0$, then $F^{t_1}\subset F^{t_2}$ by definition. Thus, it is enough to find $\bar t$ such that either $\m(F^{\bar t})>0$ or $\m(F^{-\bar t})>0$. If no such $\bar t$ exists, then Cavalieri's principle~\eqref{eq:cavalieri} implies that $\|u\|_1=0$, against our initial choice of $u$. Hence, claim~\eqref{eq:armadillo} follows, and so, up to possibly replacing $u$ with $-u$, we can suppose that $\m(F^t)>0$ for $t>\bar t$.
Now, by \cref{lem:coarea_mod}, we can rewrite~\eqref{eq:scoiattolo} as
\begin{equation}\label{eq:pantegana}
\int_\R \left(P(F^t) - (\Lambda_1(\Omega)+\varepsilon)\,\m(F^t)\right)\,\de t
\le 
0.
\end{equation}
Recalling that  $\Lambda_1(\Omega) +\varepsilon \le h_1(\Omega)$ according to our initial assumption, from inequality~\eqref{eq:pantegana}, we immediately get that
\[
0 
\le 
\int_\R \left(P(F^t) - h_1(\Omega)\, \m(F^t)\right)\,\de t
\le 
\int_\R \left(P(F^t) - (\Lambda_1(\Omega)+\varepsilon)\,\m(F^t)\right)\,\de t
\le 
0.
\]
Therefore, we must have that 
\begin{equation*}
P(F^t) - h_1(\Omega)\, \m(F^t)=0
\end{equation*}
for a.e.\ $t\in\R$.
Thus, taking into account that $F^t\subset \Omega$ for all $t\in\R$, we get that
\begin{equation}\label{eq:pandasauro}
\Lambda_1(\Omega)+\varepsilon 
\ge 
\frac{P(F^t)}{\m(F^t)} 
\ge 
h_1(\Omega)
\end{equation}
for a.e.\ $t>\bar t$, that is, $\Lambda_1(\Omega)+\varepsilon= h_1(\Omega)$ for all choices of $\varepsilon>0$ suitably small, which is clearly impossible.
Therefore, $\Lambda_1(\Omega)\ge h_1(\Omega)$, as desired.

We now conclude the proof with the case $N>1$.
Let $\varepsilon>0$ and let $\{u_i : i=1,\dots,N\}$ be a viable $N$-tuple for the definition of $\Lambda_N(\Omega)$ such that
\[
\sum_{i=1}^N\frac{\Var(u_i)}{\|u_i\|_1} 
\le \Lambda_N(\Omega) + \varepsilon.
\]
For each $i=1,\dots,N$, the function $u_i$ provides a viable competitor in the definition of $\Lambda_1(\Omega)$.
Consequently, using the inequality proved for the case $N=1$, we get that
\[
\Lambda_N(\Omega) + \varepsilon 
\ge 
N\Lambda_1(\Omega) 
= 
Nh_1(\Omega).
\]
The conclusion thus follows by letting $\varepsilon\to0^+$.
\end{proof}

We end the present section with the following simple consequence of \cref{res:relation_with_first_1_eigen}, and refer to~\cites{ACC05,BCN02} for the Euclidean case, and some final remarks.

\begin{corollary}\label{cor:eigenfunction_cheeger}
Let properties~\ref{prop:empty}, \ref{prop:lsc}, and~\ref{prop:complement_set} be in force.
Let $\Omega\in\mathscr{A}$ be a $1$-admissible set with $\m(\Omega)\in (0,+\infty)$. A function $u\in BV_0(\Omega,\m)$ attains $\Lambda_1(\Omega)$ if and only if the sets $F^t$ introduced in \cref{lem:coarea_mod} are $1$-Cheeger sets of $\Omega$ for a.e.\ $t\in\R$ such that $\m(F^t)>0$. Moreover, there exists a unique (up to a multiplicative factor) eigenfunction of the variational problem for $\Lambda_1(\Omega)$ if and only if there exists a unique $1$-Cheeger set.
\end{corollary}

\begin{proof}
Recalling that~\ref{prop:empty} and~\ref{prop:complement_set} imply the validity of~\ref{prop:space}, we can argue as in the proof of \cref{res:relation_with_first_1_eigen}, with the exception that we can now choose a minimizer $u\in BV_0(\Omega,\m)$ for $\Lambda_1(\Omega)$, with no need to work with some given $\varepsilon>0$. In particular, in place of~\eqref{eq:pandasauro}, we obtain
\[
\Lambda_1(\Omega) 
= 
\frac{P(F^t)}{\m(F^t)} 
= 
h_1(\Omega)
\]
for all $t\in\R$ such that $\m(F^t)>0$, proving the first implication. On the other hand, owing again to \cref{lem:coarea_mod}, if all the level sets $\set*{F^t:t\in\R}$ of a viable competitor $u\in BV_0(\Omega, \m)$ with positive $\m$-measure are $1$-Cheeger sets, then
\[
\Lambda_1(\Omega)\|u\|_1 
= 
h_1(\Omega)\int_\R \m(F^t)\, \de t 
= 
\int_\R P(F^t)\, \de t 
= 
\Var(u),  
\]
so that $u$ must be a minimizer attaining $\Lambda_1(\Omega)$.

Now, exploiting the first part of the claim, it is easy to show the part regarding uniqueness. First, notice that, given any $1$-Cheeger set $C$, the function $u=c \chi_C$ is a $1$-eigenfunction for any $ c\neq 0$. Thus, two different Cheeger sets must provide two different $1$-eigenfunctions. 
Conversely, if $u_1$ and $u_2$ are two distinct $1$-eigenfunctions such that $u_1\neq cu_2$ for all $c\ne0$, then we can find $\bar t\in\R$ such that the two level sets  $\set*{u_1>\bar t}$ and $\set*{u_2 > \bar t}$ have positive $\m$-measure and are distinct, hence identifying two different $1$-Cheeger sets.
\end{proof}

\begin{remark}
Let properties \ref{prop:empty}, \ref{prop:lsc}, and~\ref{prop:complement_set} be in force. Whenever a $1$-Cheeger set $C$ exists (for instance under the assumptions of \cref{res:existence_N-Cheeger}),  \cref{cor:eigenfunction_cheeger} yields the existence of eigenfunctions of the variational problem defining $\Lambda_1(\Omega)$, by setting $u= c\chi_{C}$, with $c\ne0$.
\end{remark}

\begin{remark}
In \cref{sec:ineq_M_N_constants}, we already discussed some examples of sets $\Omega$ for which $Nh_1(\Omega) = h_N(\Omega)$ (recall \cref{rem:(N+2)-barbells}). Thus, whenever \cref{res:relation_with_first_1_eigen} applies, we obtain that  $\Lambda_N(\Omega)$ equals these values for such sets $\Omega$.
\end{remark}

\begin{remark}
We point out that, in~\cite{CL19}*{Thm.~3.1}, the authors define $\Lambda_N(\Omega)$ as the infimum of a different variational problem, and prove that it coincides with $h_N(\Omega)$. The adaptation of the approach of~\cite{CL19} to our more general framework will be discussed in the forthcoming paper~\cite{SSforth}.
\end{remark}

\begin{remark}
\label{rem:relation_with_first_1_eigen_no_complement}
A key tool for the proof of \cref{res:relation_with_first_1_eigen} is the symmetric coarea formula of \cref{lem:coarea_mod}, which holds provided that also \ref{prop:complement_set} is enforced. In particular, this excludes anisotropic Euclidean settings where the Wulff shape is not central symmetric. A workaround would be to  tweak the variational problem~\eqref{eq:def_lambda_N} by instead considering
\begin{align*}
\tilde{\Lambda}_1(\Omega)&= \inf\set*{ \frac{\Var(u)}{\|u\|_1} :
u\in BV_0(\Omega,\m),\ \|u\|_1>0,\ u\ge 0
}.
\end{align*}
 In such a way, in the proof of \cref{res:relation_with_first_1_eigen}, the symmetric coarea formula would not be needed (as $u\ge 0$), and one could then establish the equality $\tilde{\Lambda}_1(\Omega) = h_1(\Omega)$.
\end{remark}

\section{Relation with first \texorpdfstring{$p$}{p}-eigenvalue and \texorpdfstring{$p$}{p}-torsion}

\label{sec:p_eigenvalue}

In this section, we work in a topological non-negative $\sigma$-finite measure space endowed with a relative perimeter functional as in \cref{ssec:relative_perimeter}.
We discuss the relations between the $1$-Cheeger constant and two other variational quantities, the first $p$-eigenvalue and the $p$-torsion function, extending to the present more general setting the results obtained in~\cites{KF03,KN08} and~\cite{BE11}. 

\subsection{Relation with first \texorpdfstring{$p$}{p}-eigenvalue}

The following definition is motivated by the strong approximation proved in \cref{res:strong_approx_weak_p_slope}.

\begin{definition}[The set $\RW^{1,p}_0(\Omega,\m)$ for $p\in(1,+\infty)$]
\label{def:sobolev_zero}
Let properties~\ref{propR:empty}, \ref{propR:space}, \ref{propR:sub-mod}, \ref{propR:lsc}, and~\ref{propR:BV_cone} be in force.
Let $p\in(1,+\infty)$ and let $\Omega\subset X$ be a non-empty open set.
We say that $u\in\RW^{1,p}_0(\Omega,\m)$ if $u\in\RW^{1,p}(X,\m)$ and there exists a sequence $\{u_k\}_{k\in\N}\subset\RW^{1,1}(X,\m)\cap L^p(X,\m)$ as in \cref{res:strong_approx_weak_p_slope} such that, in addition,
\begin{equation*}
u_k\in \mathrm{C}^0(X)\ \text{and}\ \supp u_k\subset\overline\Omega\ \text{for all}\ k\in\N.
\end{equation*}
\end{definition}

Under the validity of properties \ref{propR:empty}, \ref{propR:space}, \ref{propR:sub-mod}, \ref{propR:lsc}, and~\ref{propR:BV_cone}, since $0\in\RW^{1,1}(X,\m)$ by \cref{res:basic_props_1_slope}, we know that $0\in\RW^{1,p}_0(\Omega,\m)$.

The following definition is focused on open sets for which $\RW^{1,p}_0(\Omega,\m)\ne\set*{0}$.

\begin{definition}[$p$-regular open set and first $p$-eigenvalue]
\label{def:p_regular_open_and_first_p_eigen}
Let properties \ref{propR:empty}, \ref{propR:space}, \ref{propR:sub-mod}, \ref{propR:lsc}, and~\ref{propR:BV_cone} be in force.
Let $p\in(1,+\infty)$.
A non-empty open set $\Omega\subset X$ is \emph{$p$-regular} if $\RW^{1,p}_0(\Omega,\m)\ne\set*{0}$. 
In this case, we let 
\begin{equation*}
\lambda_{1,p}(\Omega)
=
\inf\set*{\frac{\||\nabla u|_p\|_p^p}{\|u\|_p^p} : u\in\RW^{1,p}_0(\Omega,\m),\ \|u\|_p>0}\in[0,+\infty)
\end{equation*}
be the \emph{first $p$-eigenvalue} relative to the $\RW^{1,p}$-energy on~$\Omega$. Here, $|\nabla u|_p$ denotes the weak $p$-slope of $u$ defined in Definition \ref{def:weak_p_slope}.
\end{definition}

According to \cref{def:p_regular_open_and_first_p_eigen}, if $\Omega$ is $p$-regular, then we can find a non-zero function $u\in\RW^{1,p}_0(\Omega,\m)$.
As a consequence, by \cref{def:sobolev_zero}, we can also find a function $v\in\RW^{1,1}(X,\m)\cap \mathrm{C}^0(X)$ with $\supp v\subset\overline\Omega$.
Consequently, $v\in BV_0(\Omega,\m)$ with $\|v\|_1>0$.
Therefore, the first $1$-eigenvalue $\lambda_{1,1}(\Omega)$ relative to the variation $\RVar$ on~$\Omega$ introduced in \cref{def:first_1_eigen} is well-posed.
In addition, thanks to \cref{res:relation_with_first_1_eigen}, $\lambda_{1,1}(\Omega)$ coincides with $h_1(\Omega)$ whenever $\Omega$ is $1$-admissible (with respect to the variation functional $\Var(\,\cdot\,)=\RVar(\,\cdot\,;X)$), it satisfies $\m(\Omega)\in(0,+\infty)$, and the perimeter in~\eqref{eq:def_tot_RP} satisfies~\ref{prop:complement_set}. 

As a corollary of the following result, we prove that the $1$-Cheeger constant $h_1(\Omega)$ provides a lower bound on the first $p$-eigenvalue.

\begin{theorem}[Relation with first $p$-eigenvalue]\label{thm:CheegerIn}
Let the properties \ref{propR:empty}, \ref{propR:space}, \ref{propR:sub-mod}, \ref{propR:lsc}, \ref{propR:BV_cone}, and~\ref{propR:local} be in force and let $p\in(1,+\infty)$.
If $\Omega\subset X$ is a $p$-regular open set, then 
\begin{equation}\label{eq:CheegerIn}
\lambda_{1,p}(\Omega)\ge\left(\frac{\lambda_{1,1}(\Omega)}{p}\right)^p.
\end{equation}
\end{theorem}

\begin{proof}
Let $u\in\RW^{1,p}_0(\Omega,\m)$ be such that $\|u\|_p>0$.
Then, according to \cref{def:sobolev_zero}, we can find $u_k\in\RW^{1,1}(X,\m)\cap L^p(X,\m)\cap \mathrm{C}^0(X)$ with $|\nabla u_k|\in L^p(X,\m)$ and $\supp u_k\subset\overline\Omega$ for all $k\in\N$ such that $u_k\to u$ and $|\nabla u_k|\to|\nabla u|_p$ both in $L^p(X,\m)$ as $k\to+\infty$.
Now let $\phi(r)=r|r|^{p-1}$ for all $r\in\R$ and note that $\phi\in \mathrm{C}^1(\R)$ is strictly increasing, with $\phi'(r)=p|r|^{p-1}$ for all $r\in\R$.
By \cref{res:chain_rule}, we get that $\phi(u_k)\in\RW^{1,1}(X,\m)\cap \mathrm{C}^0(X)$ with $\supp\phi(u_k)\subset\overline\Omega$ and $|\nabla \phi(u_k)|=p|u_k|^{p-1}|\nabla u_k|$ $\m$-a.e.\ in~$X$ for all $k\in\N$.
In particular, $\phi(u_k)\in BV_0(\Omega,\m)$ for all $k\in\N$. 
Thus, by H\"older's inequality, we can estimate
\begin{equation*}
\begin{split}
\RVar(\phi(u_k);X)
&=
\||\nabla\phi(u_k)|\|_1
=
p\int_X|u_k|^{p-1}|\nabla u_k|\,\de \m
\\
&
\le 
p\||u_k|^{p-1}\|_{\frac{p}{p-1}}\||\nabla u_k|\|_p
=
p\|u_k\|_p^{p-1}\||\nabla u_k|\|_p\,.
\end{split}
\end{equation*}
Thus,
\begin{equation*}
\lambda_{1,1}(\Omega)
\le
\frac{\RVar(\phi(u_k);X)}{\|\phi(u_k)\|_1}
\le
\frac{p\|u_k\|_p^{p-1}\||\nabla u_k|\|_p}{\|u_k\|_p^p}
=
\frac{p\||\nabla u_k|\|_p}{\|u_k\|_p}
\end{equation*}
for all $k\in\N$. 
Letting $k\to+\infty$, we obtain
\begin{equation*}
\lambda_{1,1}(\Omega)
\le
p\,\frac{\||\nabla u|_p\|_p}{\|u\|_p}
\end{equation*}
for $u\in\RW^{1,p}_0(\Omega,\m)$ with $\|u\|_p>0$ and the proof is complete.
\end{proof}

Assuming~\ref{prop:complement_set}, we can combine \cref{thm:CheegerIn} with \cref{res:relation_with_first_1_eigen} obtaining the following corollary.

\begin{corollary}
\label{res:relation_p_eigen_cheeger}
Let the assumptions of \cref{thm:CheegerIn} be in force. If the perimeter in~\eqref{eq:def_tot_RP} also satisfies~\ref{prop:complement_set} and $\Omega\subset X$ is $1$-admissible with respect to the variation in~\eqref{eq:def_tot_RP} with $\m(\Omega)\in(0,+\infty)$, then 
\begin{equation*}
\lambda_{1,p}(\Omega)
\ge 
\left(\frac{h_1(\Omega)}{p}\right)^p.
\end{equation*}
\end{corollary}

\subsection{Relation with \texorpdfstring{$p$}{p}-torsional creep function}

Assume properties \ref{propR:empty}, \ref{propR:space}, \ref{propR:sub-mod}, \ref{propR:lsc}, and~\ref{propR:BV_cone} to be in force.
Let $p\in(1,+\infty)$ and let  $\Omega\subset X$ be a non-empty $p$-regular open set with $\m(\Omega)<+\infty$.
We let $J_p\colon \RW^{1,p}_0(\Omega,\m)\to\R$, be
\begin{equation*}
J_p(u)
=
\frac1p\int_{\Omega}|\nabla u|_p^p\,\de\m
-
\int_{\Omega} u\,\de\m,
\end{equation*}
be the \emph{$p$-torsional creep energy functional}, see~\cite{Kaw90}. When enough structure is available on the ambient space, the Euler--Lagrange equation associated to the functional $J_p$ is given by 
\begin{equation}\label{eq:torsional_creep_PDE}
\begin{cases}
\begin{aligned}
-\Delta_p u &= 1, \qquad &&\text{in $\Omega$,}
\\
u&=0, &&\text{on $\partial \Omega$.}
\end{aligned}
\end{cases}
\end{equation}

Thanks to \cref{res:basic_props_p_relaxed_1_slope}\ref{item:basic_props_p_relaxed_1_slope_convex}, the functional $J_p$ is strictly convex on the convex set $\RW^{1,p}_0(\Omega,\m)$.
Hence, the \emph{torsional creep problem} 
\begin{equation*}
T_p(\Omega) = \inf\set*{J_p(u) : u\in \RW^{1,p}_0(\Omega,\m)}
\end{equation*}
has at most one minimizer. If this exists, we denote it by $w_p\in\RW^{1,p}_0(\Omega,\m)$. In particular, since $0\in\RW^{1,p}_0(\Omega,\m)$, we immediately see that 
\begin{equation*}
J_p(w_p)\le J_p(0)=0,
\end{equation*}
so that
\begin{equation}\label{eq:apatosauro}
\int_\Omega|\nabla w_p|_p^p\,\de\m\le p\int_\Omega w_p\,\de\m \le p\int_\Omega |w_p|\,\de\m.    
\end{equation}

Under the assumptions of \cref{res:relation_p_eigen_cheeger}, and assuming the existence of a non-trivial minimizer $w_p$ of $J_p$, we can show that the $1$-Cheeger constant of $\Omega$ provides a bound on the $L^1(X,\m)$ norm of $w_p$, in a similar fashion to~\cite{BE11}.

\begin{theorem}[Relation with the $p$-torsional creep function]\label{thm:torsional}
Let the properties \ref{propR:empty}, \ref{propR:space}, \ref{propR:sub-mod}, \ref{propR:lsc}, \ref{propR:BV_cone},  \ref{propR:local}, and~\ref{prop:complement_set} be in force and let $p\in(1,+\infty)$.
If $\Omega\subset X$ is a $p$-regular open set which is also $1$-admissible with respect to the variation in~\eqref{eq:def_tot_RP} with $\m(\Omega)\in(0,+\infty)$, and $J_p$ has a non-trivial minimizer $w_p \neq 0$, then
\begin{equation}\label{eq:Cheeger_torsional_creep}
h_1(\Omega) \le p^{1+\frac 1p} \left(\frac{\m(\Omega)}{\|w_p\|_1} \right)^{\frac{p-1}{p}}.
\end{equation}
\end{theorem}

\begin{proof}
Using \cref{res:relation_p_eigen_cheeger}, the (non-trivial) torsional creep function $w_p$ as a competitor for $\lambda_{1,p}(\Omega)$, the inequality~\eqref{eq:apatosauro} and H\"older's inequality, we get
\begin{align*}
    \left(\frac{h_1(\Omega)}{p} \right)^p 
    &
    \le 
    \lambda_{1,p}(\Omega) 
    \le 
    \frac{\int_\Omega |\nabla w_p|_p^p\, \de \m}{\int_\Omega |w_p|^p\, \de \m}
    \\
    &
    \le
    \frac{p\int_\Omega |w_p|\, \de \m}{\int_\Omega |w_p|^p\, \de \m}
    \le
    \frac{p \m(\Omega)^{p-1}}{\left(\int_\Omega |w_p|\, \de \m\right)^{p-1}}.
\end{align*}
Rearranging, the claimed inequality follows.
\end{proof}

\begin{remark}
If the weak formulation of the torsional creep PDE~\eqref{eq:torsional_creep_PDE} is available, then one can test it against the solution $w_p$ itself, finding that
\[
\int_\Omega |\nabla w_p|_p^p\, \de \m = \int_\Omega w_p\, \de \m.
\]
Using this equality in place of inequality~\eqref{eq:apatosauro}, one gets the analog of~\eqref{eq:Cheeger_torsional_creep} with a prefactor of $p$ in place of $p^{1+\frac1p}$. Under additional structural hypotheses on the space, which allow to identify the $p$-slope with the $p$-th power of the $1$-slope, one can altogether remove the prefactor from the inequality, similarly to~\cite{BE11}*{Thm.~2}.
\end{remark}

\begin{remark}
In the statement of \cref{thm:torsional}, we need to assume that $J_p$ has a minimizer (which in case is the unique one, by strict convexity) and that this is non-trivial. This can be ensured in suitable spaces, where a Poincar\'e inequality holds, allowing to see that $J_p$ is coercive.
\end{remark}

\section{Applications}
\label{sec:examples}

In this section, we apply the general results presented above to specific settings. 
In each of the following examples, we will consider the natural topology of the ambient space. 

\subsection{Metric-measure spaces}
\label{subsec:mms}

Let $(X,d)$ be a complete and separable metric space and let $\m$ be a non-negative Borel measure (on the $\sigma$-algebra induced by the distance $d$) that is finite on bounded Borel sets and such that $\supp\m = X$.
In particular, $\m$ is a $\sigma$-finite measure on~$X$. 

Given $u\colon X\to\R$, we define the \emph{slope} of $u$ (also called the \emph{local Lipschitz constant} of $u$) the function $|\nabla u|\colon X\to[0,+\infty]$ defined as
\begin{equation*}
|\nabla u|(x)
=
\limsup_{y\to x}\frac{|u(y)-u(x)|}{d(y,x)}
\end{equation*}
for all $x\in X$. 

For an open set $A\subset X$, we say that $u\in\Lip_{\loc}(A)$ if for each $x\in\Omega$, there is $r>0$ such that $B_r(x)\subset A$ and the restriction $u|_{B_r(x)}$ is a Lipschitz function, where $B_r(x)\subset X$ denotes the $d$-ball centered at $x\in X$ with radius $r\in(0,+\infty)$.

In the present metric-measure setting, one has the following natural definition of $BV$ functions, see~\cites{ADiM14,ADiMG17, Miranda03}.

\begin{definition}[$BV$ functions in metric spaces]
\label{def:BV_in_metric_spaces}

We say that $u\in BV(X,d,\m)$ if $u\in L^1(X,\m)$ and there exists a sequence $\{u_k\}_{k\in\N}\subset\Lip_{\loc}(X)$ such that 
\begin{equation*}
u_k\to u\ \text{in}\ L^1(X,\m)
\quad
\text{and}
\quad
\sup_{k\in\N}\int_{X}|\nabla u_k|\,\de\m
<+\infty.
\end{equation*}
Moreover, we let 
\begin{equation}
\label{eq:def_variation_metric}
|Du|(A)
=
\inf
\set*{
\liminf_{k\to+\infty}\int_X|\nabla u_k|\,\de\m
: 
u_k\in\Lip_{\loc}(A),
\
u_k\to u\ \text{in}\ L^1(A,\m)
}
\end{equation}
be the \emph{variation of $u$ relative to the open set $A\subset X$}.
As usual, if $u=\chi_E$ for some measurable set $E\subset X$, then we let $P(E;A)=|D\chi_E|(A)$ be the \emph{perimeter of $E$ relative to $A$}.
\end{definition}

If $u\in BV(X,d,\m)$, then the set function $A\mapsto|Du|(A)$ is the restriction to open sets of a finite Borel measure, for which we keep the same notation.
This result was originally proved in~\cite{Miranda03} for locally compact metric spaces, and then generalized to the possibly non-locally compact setting in~\cite{ADiM14}.
Actually, as done in~\cite{Miranda03}, the convergence in $L^1(A,\m)$  in~\eqref{eq:def_variation_metric} may be replaced with the convergence in $L^1_{\loc}(A,\m)$ without affecting the overall approach. 

In the present setting, the (total) perimeter functional $P(E)=|D\chi_E|(X)$ given by~\eqref{eq:def_variation_metric} in \cref{def:BV_in_metric_spaces}  satisfies the properties \ref{prop:empty}, \ref{prop:space}, \ref{prop:sub-mod}, \ref{prop:lsc}, \ref{prop:compact}, and~\ref{prop:complement_set}.
Indeed, properties~\ref{prop:empty}, \ref{prop:space}, and~\ref{prop:complement_set} immediately follow from \cref{def:BV_in_metric_spaces}.
For property~\ref{prop:sub-mod}, we refer to~\cite{Miranda03}*{Prop.~4.7(3)}. 
Finally, property~\ref{prop:lsc} is a consequence of~\cite{Miranda03}*{Prop.~3.6} and property~\ref{prop:compact} follows from~\cite{Miranda03}*{Thm.~3.7}.

For what concerns the variation measure introduced in~\eqref{eq:def_variation_metric}, from~\cite{Miranda03}*{Prop.~4.2} and the discussion in~\cite{ADiM14}*{Sect.~1}, we can infer that
\begin{equation*}
|Du|(A)
=
\int_{\R} |D\chi_{\set*{u>t}}|(A)\,\de t
=
\int_{-\infty}^0|D\chi_{\set*{u<t}}|(A)\,\de t
+
\int_0^{+\infty}|D\chi_{\set*{u>t}}|(A)\,\de t
\end{equation*}
whenever $u\in BV(X,d,\m)$, for every Borel set $A\subset X$.

In virtue of the properties listed above, we easily deduce the validity of the relation between the $1$-Cheeger constant of an open set $\Omega\subset X$ with $\m(\Omega)\in(0,+\infty)$ and the first $1$-eigenvalue as in  \cref{res:relation_with_first_1_eigen}, meaning that 
\begin{equation*}
h_1(\Omega)
=
\inf\set*{
\frac{|Du|(X)}{\|u\|_1} 
:
u\in BV(X,d,\m),\ u|_{X\setminus\Omega}=0,\ \|u\|_1>0
}.
\end{equation*}

Incidentally, we refer the reader to~\cite{ADiM14}*{Sect.~6} for the definition of the \emph{$1$-Laplacian operator} in this general context. 

Concerning the isoperimetric-type property~\ref{prop:isoperim}, we can state the following result (inspired by~\cite{AFP00book}*{Thm.~3.46}).
Notice that inequality~\eqref{eq:mira_isop} serves as a
prototype in the present setting.
In fact, as discussed in the examples below, if the ambient metric-measure space $(X,d,\m)$ has a richer structure, then finer results are available.
Nonetheless, the isoperimetric-type inequality~\eqref{eq:mira_isop} is sufficient to replace~\ref{prop:isoperim} in the present context. 

\begin{proposition}
\label{res:mira_isop}
Let $(X,d,\m)$ be a geodesic Poincaré metric-measure space as in~\cite{Miranda03}*{Def.~2.5}.
Let $\Omega\subset X$ be an open set with $\m(\Omega)\in(0,+\infty)$.
Assume that there exists a countable family of open balls $B^i=B_{r_i}(x_i)$, $i\in\N$, with the following properties: 
\begin{enumerate}[label=(\roman*)]
\item
\label{item:covering} 
$\Omega\subset\bigcup_{i\in\N}B^i$;
\item
\label{item:overlap} 
there exists $N\in\N$ such that
$\sum_{i\in\N}\chi_{B^i}\le N$;
\item
\label{item:ahlfors}
there exists $c_0>0$ such that 
$\m(B^i)\ge c_0r_i^Q$ for all $i\in\N$, where $Q>0$ is the \emph{homogeneous dimension} as in~\cite{Miranda03}*{Rem.~2.2};
\item
\label{item:radii_inf}
$\eps_0=\inf\set*{r_i:i\in\N}>0$.
\end{enumerate}
Then, there exists a constant $C>0$ such that
\begin{equation}
\label{eq:mira_isop}
\m(E)^{\frac{Q-1}Q}\le C\,P(E)
\end{equation}
whenever $E\subset\Omega$ is an $\m$-measurable set with $\m(E)<\frac{c_0}{2}\eps_0^Q$.
\end{proposition}

\begin{proof}
Due to~\ref{item:radii_inf}, we can estimate
\begin{equation*}
\m(E\cap B^i)
\le 
\m(E)
<\frac{c_0}{2}\eps_0^Q
\le 
\frac12\m(B^i)
\end{equation*}
for all $i\in\N$.
Consequently, by~\cite{Miranda03}*{Thm.~4.5} (since $(X,d)$ is geodesic by assumption) and~\ref{item:ahlfors}, we get 
\begin{equation*}
\begin{split}
\m(E\cap B^i)^{\frac{Q-1}{Q}}
&=
\min\set*{\m(E\cap B^i),\ \m(E^c\cap B^i)}^{\frac{Q-1}{Q}}
\\
&\le
\frac{cr_i}{\m(B^i)^{1/Q}}\,P(E;B^i)
\le 
c\,c_0^{-1/Q}P(E;B^i)
\end{split}
\end{equation*}
for all $i\in\N$, where $c>0$ is the constant appearing in~\cite{Miranda03}*{Thm.~4.5}.
Hence, thanks to~\ref{item:covering} and~\ref{item:overlap}, we conclude that
\begin{equation*}
\m(E)^{\frac{Q-1}{Q}}
\le 
\sum_{i\in\N}\m(E\cap B^i)^{\frac{Q-1}{Q}}
\le 
c\,c_0^{-1/Q}\sum_{i\in\N}P(E;B^i)
\le 
c\,c_0^{-1/Q}N\,P(E),
\end{equation*}
yielding~\eqref{eq:mira_isop} with $C=c\,c_0^{-1/Q}N$.
\end{proof}

\begin{remark}
Note that, if $(X,d,\m)$ is as in the statement of \cref{res:mira_isop} and $\Omega\subset X$ is a bounded open set such that $\m(\Omega)>0$, then $\overline{\Omega}$ is a compact set, and thus, for any $r>0$, we can find $N(r)\in\N$ open balls $B_r(x_i)$, $x_i\in\overline{\Omega}$, $i=1,\dots,N(r)$, satisfying~\ref{item:covering}, \ref{item:overlap} with $N=N(r)$ (in fact, the assumption that $(X,d)$ is geodesic is not needed), and~\ref{item:radii_inf} with $\eps_0=r$. 
The validity of~\ref{item:ahlfors} holds thanks to~\cite{Miranda03}*{Eq.~2 in Rem.~2.2}, since $\overline{\Omega}\subset B_R(\bar x)$ for some $\bar x\in\Omega$ and some $R\in(0,+\infty)$.
Consequently, \cref{res:mira_isop} always applies to bounded open sets $\Omega\subset X$ with positive measure.
\end{remark}

In the metric-measure framework, the definition of  $W^{1,1}(X,d,\m)$ is not completely understood, see the discussion in~\cite{ADiM14}*{Sect.~8}.
As usual, one possibility is to say that $u\in W^{1,1}(X,d,\m)$ if $u\in BV(X,d,\m)$ and $|Du|\ll\m$, and then to proceed as in \cref{subsec:sobolev_1_space} in order to work out the machinery needed to establish the relation between the $1$-Cheeger constant and the first $p$-eigenvalue in \cref{res:relation_p_eigen_cheeger}. 
However, one can exploit the richer structure of the ambient space to get a more direct and plainer approach to the relation with the first $p$-eigenvalue.
Let us briefly detail the overall idea. 
In the spirit of~\cite{Cheeger99} and in an analogous way of \cref{def:BV_in_metric_spaces} (see the discussion at the end of~\cite{Miranda03}*{Sect.~2}), we say that $u\in W^{1,p}(X,d,\m)$ for some $p\in(1,+\infty)$  if there exists a sequence $\{u_k\}_{k\in\N}\subset\Lip_{\loc}(X)$ such that 
\begin{equation}
\label{eq:approx_Lip_p_def_sob}
u_k\to u\ \text{in}\ L^p(X,\m)
\quad
\text{and}
\quad
\sup_{k\in\N}\int_{X}|\nabla u_k|^p\,\de\m
<+\infty.
\end{equation} 
Therefore, we can consider the \emph{Cheeger $p$-energy} of~$u$, defined by
\begin{equation}
\label{eq:def_p-Cheeger_energy}
\mathrm{Ch}_p(u)
=
\inf
\set*{
\liminf_{k\to+\infty}\frac1p\int_X|\nabla u_k|^p\,\de\m
: 
u_k\in\Lip_{\loc}(X),
\
u_k\to u\ \text{in}\ L^p(X,\m)
},
\end{equation}
as the natural replacement of the Dirichlet $p$-energy in this  framework. 

Accordingly, for a given non-empty open set $\Omega\subset X$, we say that $u\in W^{1,p}_0(\Omega,d,\m)$ if there exists a sequence $\{u_k\}_{k\in\N}\subset\Lip_{\loc}(X)$ as in~\eqref{eq:approx_Lip_p_def_sob} (so that, in particular, $u\in W^{1,p}(X,d,\m)$) with the additional property that $\supp u_k\Subset\Omega$ for all $k\in\N$.
Therefore, coherently with what was done in \cref{def:p_regular_open_and_first_p_eigen}, we let 
\begin{equation}
\label{eq:first_p-eigen_metric}
\lambda_{1,p}(\Omega,d,\m)
=
\inf\set*{
\frac{p\,\mathrm{Ch}_p(u)}{\|u\|_p^p}
:
u\in W^{1,p}_0(X,d,\m),\ \|u\|_p>0
}.
\end{equation}

Now, it is not difficult to see that, in virtue of the definition of the Cheeger $p$-energy in~\eqref{eq:def_p-Cheeger_energy}, the infimum in~\eqref{eq:first_p-eigen_metric} can be actually restricted to functions $u\in L^p(X,\m)\cap \Lip_{\loc}(X)$ such that 
\begin{equation}
\label{eq:codirosso}
\int_X|\nabla u|^p\,\de\m<+\infty
\quad
\text{and}
\quad
\supp u\Subset\Omega.
\end{equation}
Now, if $u\in L^p(X,\m)\cap \Lip_{\loc}(X)$ satisfies~\eqref{eq:codirosso}, then the function $v=u|u|^{p-1}$ is such that $v\in L^1(X,\m)\cap\Lip_{\loc}(X)$ with $\supp v\Subset\Omega$ and
\begin{equation*}
|\nabla v|\le 
p|u|^{p-1}|\nabla u|.
\end{equation*}
Consequently, by the definition in~\eqref{eq:def_variation_metric} and H\"older's inequality, we can estimate
\begin{equation*}
|Dv|(X)
\le 
\int_X|\nabla v|\,\de\m
\le 
p\int_X|u|^{p-1}|\nabla u|\,\de\m
\le 
p\|u\|_p^{p-1}\||\nabla u|\|_p.
\end{equation*}
Therefore,
\begin{equation}
\label{eq:pigliamosche}
h_1(\Omega)
\le 
\frac{|Dv|(X)}{\|v\|_1}
\le 
\frac{p\|u\|_p^{p-1}\||\nabla u|\|_p}{\|u\|_p^p}
=
\frac{p\||\nabla u|\|_p}{\|u\|_p}\,,
\end{equation}
and thus,
\begin{equation*}
\lambda_{1,p}(\Omega,d,\m) \ge \left(\frac{h_1(\Omega)}{p}\right)^p
\end{equation*}
by the arbitrariness of $u$ in the right-hand side of~\eqref{eq:pigliamosche}, proving \cref{res:relation_p_eigen_cheeger}.

Similar considerations can be done for the relation of the $1$-Cheeger constant with the $p$-torsional creep function as in \cref{thm:torsional}. 
We leave the analogous details to the interested reader.

\subsection{Euclidean spaces with density}

Let $\mathscr A=\mathscr{B}(\mathbb{R}^n)$ be the Borel $\sigma$-algebra in $\R^n$
and consider two lower-semicontinuous density functions 
\[g\in L^\infty(\R^n;[0,+\infty)),\quad\text{and}\quad
h\in L^\infty(\R^n\times{\mathbb S}^{n-1};(0,+\infty)),\] 
so that $h$ is convex in the second variable and locally bounded away from zero, i.e., for any bounded set $\Omega\subset\R^n$, there exists $C>0$ such that
\begin{equation}
\label{eq:hbound}\frac{1}{C}\leq h(x,\nu)\leq C, \quad \forall x\in\Omega,\ \forall\nu\in\mathbb S^{n-1}\,.
\end{equation}
For any $E\in\mathscr A$, we let the weighted volume and perimeter of $E$ to be defined, respectively, by 
\begin{equation*}
\m_g(E) = \int_E g(x)\,\de x\,,
\end{equation*}
\begin{equation*}
P_h(E) = \begin{cases}
\displaystyle \int_{\partial^* E} h(x,\nu_E(x))\, \de\mathcal{H}^{n-1}(x),& \text{if }\chi_E\in BV_{\mathrm{loc}}(\R^n),\\[1.5ex]
+\infty\,,&\text{otherwise}.
\end{cases}
\end{equation*}
Here, we let $\partial^* E$ be the reduced boundary of $E$ and, for every $x\in\partial^* E$, $\nu_E(x)\in\mathbb S^{n-1}$ be the outer unit normal vector to $E$ at $x$ (see~\cite{AFP00book} for definitions and properties of sets of finite perimeter). 

If $g$ and $h$ are identically equal to $1$, $\m_g=\mathscr L^n$ is the Lebesgue measure, and $P_h=P_{\rm Eucl}$ is the standard Euclidean perimeter that satisfies all properties \ref{prop:empty}--\ref{prop:isoperim} and \ref{prop:complement_set} in the measure space $(\R^n,\mathscr A,\m)$. In particular, \ref{prop:isoperim} holds with $f(\varepsilon)=n\omega_n^{1/n}\varepsilon^{-1/n}$, and follows from the standard isoperimetric inequality
\begin{equation}\label{eq:isop}
    P_{\rm Eucl}(E)\geq n\omega_n^{1/n} \mathscr L^n(E)^{1-\frac{1}{n}},
\end{equation}
holding for any $E\in \mathscr A$.
The Cheeger problem in this setting is standard, see \cites{Leo15, Parini11}, and its minimizers are now completely characterized for a large class of planar sets~\cites{Can22, KF03, LNS17, LS20, Saracco18}, and reasonably well-understood for convex $N$-dimensional bodies~\cites{ACC05, BP21}. Recently, Cheeger clusters have been introduced and studied in~\cite{Caroccia17}, see also~\cites{BFVV18, BF19, CL19}, and in~\cites{BP18, SSforth} in relation to more general combinations than the sum of their Cheeger constants. Interestingly, the Euclidean Cheeger problem plays a role in the ROF model for regularization of noisy images, as highlighted in~\cite{ACC05}, see also~\cite{Leo15}*{Sect.~2.3}, and this is also linked to our \cref{sec:PMC} and to our \cref{cor:eigenfunction_cheeger}.

We now turn to the case of volume and perimeter with general densities, for which the Cheeger problem has been considered for $N=1$,  e.g., in \cites{ILR05,Saracco18,LT19,CFM09, KN08}.
We discuss properties \ref{prop:empty}--\ref{prop:complement_set} for the general densities $h,g$ above. 
Notice that \eqref{eq:hbound} implies that sets with locally finite  perimeter $P_h$ are all and only those with locally finite Euclidean perimeter.
Properties \ref{prop:empty} and \ref{prop:space} are immediate from the definitions. Given $E,F\in\mathscr A$, \ref{prop:sub-mod} follows from the validity of the following relations between reduced boundaries and sets operations:
\[
\begin{split}
    \partial^*(E\cap F)&\subseteq (F^{(1)}\cap \partial^*E)\cup (E^{(1)}\cap \partial^*F),\\ \partial^*(E\cup F)&\subseteq (F^{(0)}\cap \partial^*E)\cup (E^{(0)}\cap \partial^*F).
\end{split}
\]
Here, $E^{(t)}$ denotes the set of points of density $t\ge 0$ for $E$ and we recall that $\mathscr L^n(\R^n\setminus (E^{(0)}\cup E^{(1)}))=0$, see, e.g., \cite{Maggi}*{Thm.~16.3} for more details.
Property \ref{prop:lsc} holds true thanks to Reshetnyak lower-semicontinuity Theorem, see \cite{AFP00book}*{Thm.~2.38}, and uses the lower-semicontinuity and the convexity assumptions on $h$. Property \ref{prop:compact} follows from the standard compactness theorem for sets with equibounded perimeter and from \eqref{eq:hbound}. Property \ref{prop:complement_set} is equivalent to assume that $h$ is even in the second variable. 

Finally, we discuss \ref{prop:isoperim}. If \eqref{eq:hbound} also holds on $\Omega=\R^n$, the isoperimetric inequality \eqref{eq:isop} extends to this setting by observing that, for any $\varepsilon>0$ and $E\in\mathscr A$ such that $\m(A)\leq\varepsilon$, we have
\begin{equation}\label{eq:isop-density}
\begin{split}
P_h(E)&\ge \frac{P_{\rm Eucl}(E)}{C}\ge \frac{n\omega_n^{1/n}\mathscr L^n(E)^{1-\frac{1}{n}}}{C}\\
&\ge 
\frac{n\omega_n^{1/n}\m_g(E)^{1-\frac{1}{n}}}{C(\sup g)^{\frac{n-1}{n}}} 
\ge 
f(\varepsilon) \m_g(E),
\end{split}
\end{equation}
with $f(\varepsilon)=n\omega_n^{1/n}/[C\varepsilon^{1/n}(\sup g)^{\frac{n-1}{n}}]$. 
In this case, 
\cref{res:existence_N-Cheeger} implies existence of Cheeger $N$-clusters of any admissible set $\Omega\subset\R^n$ for perimeter and volume with double (anisotropic) densities. Observe that, in order for $\Omega$ to be admissible, it is needed that $g>0$ on a Borel set of positive measure contained in $\Omega$.
This covers the existence results already present in the literature for $N=1$ \cites{CFM09,KN08,Saracco18} and double densities, or for Euclidean densities and $N>1$ \cites{Caroccia17,CL19}, and generalize it to the case of double densities and $N>1$.

In the general case where \eqref{eq:hbound} does not extend to a global bound, \ref{prop:isoperim} might not hold. Some specific examples of this type are discussed in the following subsections.
Nonetheless, \cref{prop:min_PMC} applies to general densities, establishing the relation between the $1$-Cheeger constant with the curvature functional, as previously discussed, e.g., in \cites{CFM09,ABT15}.
Assuming the symmetry assumption \ref{prop:complement_set}, \cref{res:relation_with_first_1_eigen} shows that the $1$-Cheeger constant $h_1(\Omega)$ corresponds to the first $1$-eigenvalue $\lambda_{1,1}(\Omega)$ of~\cref{def:first_1_eigen}.

Estimates of the first $p$-eigenvalue in terms of the $1$-Cheeger constant are proved in \cite{KN08} for anisotropic (symmetric) perimeters whose density does not depend on the position, and Euclidean volume. 
In this setting, we observe that $P_h$ admits the following distributional formulation:
\[
P_h(E;A)=\sup\set*{\int_E\div\varphi\,\de x : \varphi\in \mathrm{C}^1_c(A; \R^n),\ \sup_{x\in A}h^*(\varphi)\leq 1},
\]
where $h^*$ denotes the norm dual to $h$. The latter yields the validity of \ref{propR:local} and \ref{propR:BV_cone}, thus allowing to apply \cref{thm:CheegerIn} and to establish a relation between the $1$-Cheeger constant and the first $p$-eigenvalue of \cref{def:p_regular_open_and_first_p_eigen}, in the same spirit of~\cite{KN08}.
As far as we know, the relation between the $1$-Cheeger constant for more general densities $h=h(x,\nu)$ and the spectrum of specific ``$p$-Laplace'' operators in the spirit of~\cite{KN08} is an open question.

\subsubsection{Gaussian space} When \eqref{eq:hbound} does not extend to  $\Omega=\R^n$, property~\ref{prop:isoperim} cannot be deduced as in \eqref{eq:isop-density}. A specific setting where this happens is the Gaussian space, corresponding to the choice $g=h=\gamma$, where
\begin{equation}\label{defgamma}
\gamma(x)=\frac{1}{(2\pi)^{n/2}}e^{-\frac{\|x\|^2}{2}}\,.
\end{equation}
Let us notice that \eqref{eq:hbound} only holds locally 
and that $(\mathbb{R}^n, \m_\gamma)$ is a probability space, i.e.,
\[
\m_\gamma(\mathbb{R}^n)=\frac{1}{(2\pi)^{n/2}}\int_{\mathbb{R}^n}e^{-\frac{\|x\|^2}{2}}\, \de x=1\,.
\]
Similarly to the Euclidean case, we can define the Gaussian perimeter of a Borel set $E$ inside an open set $\Omega\subset \mathbb{R}^n$ as
\begin{equation}\label{def}
P_\gamma(E;\Omega)=\sup\set*{\int_E(\div\varphi-\varphi\cdot x)\, \de\m_\gamma(x)\ :\ \varphi\in \mathrm{C}^{\infty}_c(\Omega;\mathbb{R}^n),\ \|\varphi\|_{\infty}\leq 1}.
\end{equation}
It is easy to see that if a set $E$ has finite Gaussian perimeter, then it also has locally finite Euclidean perimeter and
\begin{equation*}
  P_\gamma(E)  =\frac{1}{(2\pi)^{n/2}}\int_{\partial^{*} E}e^{-\frac{\|x\|^2}{2}}\, \de \mathcal{H}^{n-1}(x).
\end{equation*}
Properties \ref{prop:empty}--\ref{prop:compact} and \ref{prop:complement_set}  follow as above. 
We discuss the isoperimetric property~\ref{prop:isoperim}. For every Borel set $E\subset \mathbb{R}^n$, the following isoperimetric inequality holds (see \cites{BBJ17,Bor75,SC74})
\begin{equation*}
    P_\gamma(E)\geq \mathcal{U}(\m_\gamma(E)),
\end{equation*}
where $\mathcal{U}\colon\mathbb{R}\to\mathbb{R}$ is the \emph{Gaussian isoperimetric function} defined as $\mathcal{U}(t)=\Phi'\circ \Phi^{-1}(t)$, $t\in\mathbb{R}$, with
\[
\Phi(t)=\frac{1}{\sqrt{2\pi}}\int_{-\infty}^t e^{-\frac{s^2}{2}}\, \de s,
\]
and it has the following asymptotic behavior \cite{CMN10}:
\begin{equation*}
    \lim_{s\to 0}\frac{\mathcal{U}(s)}{s|\ln(s)|^{\frac{1}{2}}}=1.
\end{equation*}
We deduce \ref{prop:isoperim} by setting $f\colon(0,+\infty)\to (0,+\infty)$ as
\[
f(\varepsilon)=\frac{\mathcal{U}(\varepsilon)}{\varepsilon}.
\]
Moreover, using the distributional formulation \eqref{def}, one can deduce the validity of \ref{propR:BV_cone} and \ref{propR:local}. Therefore, all of our results apply in this setting. While the existence of $1$-Cheeger sets was already known, see~\cites{CMN10, JS21}, and the clustering isoperimetric problem has been recently addressed in~\cite{MN22}, the Cheeger cluster problem for $N>1$ had never been treated. The relation with the prescribed curvature functional had been studied in~\cite{CMN10}, while, up to our knowledge, the relation with the first eigenvalue of the Dirichlet $p$-Laplacian had never been proved, but only quickly observed in~\cite{JS21} for $p=2$ (the Ornstein--Uhlenbeck operator).

\subsubsection{Monomial and radial weights} 
Further settings where \ref{prop:isoperim} cannot be deduced from~\eqref{eq:hbound} are those of monomial and radial densities. 
Given $A=(a_1,\dots,a_n)\in\R^n$, with $a_i\geq 0$, for  $i=1,\dots,n$, a monomial weight is $g(x)=x^A$, where we used the notation $x^A=|x_1|^{a_1}\cdots|x_n|^{a_n}$.
A radial weight, instead, is of the type $g(x)=\|x\|^q$, $q\ge 0$. 

We have already shown that  properties \ref{prop:empty}--\ref{prop:compact} and \ref{prop:complement_set}  hold true. The isoperimetric property is discussed (for Lipschitz sets) for monomial weights in \cites{CRO13,AF22,ABCMP21,CGPRS20, CROS12}, and for radial weights in~\cites{Csato15,Csato17,ABCMP17}.
The Cheeger problem in the monomial setting has also been considered in~\cites{BP21,ABCMP21}.

\subsection{Non-local perimeters} In this section, we discuss applications to non-local perimeters. 

\subsubsection{Classical non-local perimeters} We focus on the non-local perimeters considered in \cites{CN18,RossiJ,Valdi}. Let $K:\R^n\to(0,+\infty)$ be such that
\begin{equation}\label{eq:nl-ass}
\min\{|x|,1\}K(x)\in L^1(\R^n)\quad\text{and}\quad K(x)=K(-x)\,\ \forall x\in\R^n.
\end{equation}
For a measurable set $E\subset\R^n$, we let 
\begin{equation}
\label{eq:nonlocal_per}
P_K(E)=\frac12\int_{\R^n}\int_{\R^n}K(x-y)|\chi_E(x)-\chi_E(y)|\,\de y\,\de x.
\end{equation}

We now discuss the properties of the $K$-perimeter on the measure space $(\R^n,\mathscr B(\R^n),\mathscr{L}^n)$.
Under the general assumptions~\eqref{eq:nl-ass}, properties \ref{prop:empty}, \ref{prop:space}, and~\ref{prop:complement_set}
are direct consequences of \eqref{eq:nonlocal_per}, whereas for the validity of \ref{prop:sub-mod} and \ref{prop:lsc}, we refer to~\cite{CN18}*{Prop.~2.2}.
As a consequence, \cref{res:relation_with_first_1_eigen} applies, establishing the link between $\lambda_{1,1}(\Omega)$ and $h_1(\Omega)$, already proved in~\cite{BLP14} for the fractional $s$-perimeter and in~\cite{MRT19}*{Thm.~5.3} for kernels satisfying $K\in L^1(\R^n)$.

The validity of~\ref{prop:compact} holds provided that, besides~\eqref{eq:nl-ass}, one assumes that $K\in L^1(\R^n\setminus B_r)$ for all $r>0$, see~\cite{BSu}*{Thm.~2.11}.
This yields \cref{res:existence_N-Cheeger} and \cref{prop:min_PMC}, also see the discussion in~\cite{BSu}*{Sect.~3}. 

The isoperimetric property \ref{prop:isoperim} holds, provided that $K$ is radially symmetric decreasing, see~\cite{CN18}*{Prop.~3.1} and~\cite{BSu}*{Thm.~2.19}.
Moreover, the function 
\[
f(\varepsilon)=\frac{P_{K}(B^\varepsilon)}{\varepsilon},\quad \varepsilon>0\,,
\]
where $B^\eps$ is the Euclidean ball centered at the origin with volume $\varepsilon$,
satisfies $f(\eps)\to +\infty$ as ${\eps\to0^+}$, provided that $K\notin L^1(\R^n)$, see~\cite{CN18}*{Lem.~3.2} and~\cite{BSu}*{Lem.~2.22}, thus yielding~\ref{prop:isoperim}.

For a study of the Cheeger problem for the non-local perimeter functional $P_K$ and the (weighted) Lebesgue measure, as well as for the relation between this non-local Cheeger problem with the associated non-local $L^1$ denoising model and the prescribed mean curvature functional, see~\cite{BSu}*{Sects.~3 and~6.3}.

The most relevant example of non-local perimeter functional satisfying all the aforementioned properties is the \emph{fractional  $s$-perimeter}, $s\in(0,1)$,
\begin{equation}
\label{eq:fractional_per}
P_{s}(E)=\int_{\R^n}\int_{\R^n}\frac{|\chi_E(x)-\chi_E(y)|}{|x-y|^{n+s}}\,\de y\,\de x,
\end{equation}
corresponding to the choice $K_s(x)={|x|^{-n-s}}$.
We mention that the Cheeger problem for $P_s$ has been introduced and studied in \cite{BLP14}, where 
existence of $N$-Cheeger sets of a bounded open set $\Omega\subset\R^n$ is proved for $N=1$. Up to our knowledge, for $N>1$, the $N$-Cheeger problem has never been considered in this setting, whereas the clustering isoperimetric problem has been treated in~\cite{CN20}.
We also refer to~\cite{Bessas22}*{Thm.~1.5} for a discussion of the fractional Cheeger constant with respect to the existence of minimizers of the prescribed mean curvature functional. 

\subsubsection{Fractional Gaussian spaces}

For $x,y\in\R^n$, $t>0$, let $M_t(x,y)\ge0$ be the \emph{Mehler kernel}, see \cites{CClMP20, CClMP21, CClMP21a} for the precise definition. For $s\in(0,1)$, we set
\[
K_{\sigma}(x,y)=\int_{0}^{+\infty} \frac{M_t(x,y)}{t^{\frac{s}{2}+1}}\, \de t\,,
\]
and we let
\begin{equation}
\label{eq:frac_gauss_per}
P_s^\gamma(E)=\int_{E} \int_{E^c}K_s(x,y)\, \de\gamma(y)\, \de\gamma(x)
\end{equation}
be the \emph{fractional Gaussian $s$-perimeter} of the measurable set $E\subset\R^n$,
where the measure $\gamma$ is as in \eqref{defgamma}.
Properties \ref{prop:empty} and \ref{prop:space} directly follow from the definition, whereas~\ref{prop:complement_set} is a consequence of the symmetry of the kernel $K_s(x,y)$.
For the validity of~\ref{prop:lsc} and~\ref{prop:compact}, we refer to \cites{CClMP20, CClMP21, CClMP21a}.
Therefore, \cref{prop:min_PMC} and \cref{res:relation_with_first_1_eigen} hold true. 
Property~\ref{prop:sub-mod} can be proved exactly as in~\cite{CN18}*{Prop.~2.2}. 

Concerning~\ref{prop:isoperim}, the following isoperimetric inequality 
\begin{equation*}
P_s^\gamma(E)\geq I_s^\gamma(\gamma(E))
\end{equation*}
is proved for every measurable set $E\subset\mathbb{R}^n$ in \cite{NPS18}.
Here $I_s^\gamma\colon(0,1)\to (0,+\infty)$ is the \emph{fractional Gaussian isoperimetric function}, i.e.,  $I_s^\gamma(v)$ is the fractional Gaussian $s$-perimeter of any halfspace $H$ such that $\gamma(H)=v$. As far as we know, the asymptotic behavior of $I_s(v)$ as $v\to 0^+$ is not known, and hence, \ref{prop:isoperim} cannot be guaranteed. 

\begin{remark}
Let $(X,\mathscr A,\m)$ be a non-negative $\sigma$-finite measure space and let $K\colon X\times X\to[0,+\infty]$ be a symmetric $(\m\otimes\m)$-measurable function.
For $A,B\in\mathscr A$, we set
\begin{equation*}
L_{K,\m}(A,B)=\int_A\int_B K(x,y)\,\de\m(x)\,\de\m(y)
\end{equation*}
and, given $E,\Omega\in\mathscr A$, we let
\begin{equation*}
\begin{split}
P_{K,\m}(E;\Omega)
&=
L_{K,\m}(E\cap\Omega,E^c\cap\Omega)
+
L_{K,\m}(E\cap\Omega,E^c\cap\Omega^c)
\\
&\quad+
L_{K,\m}(E\cap\Omega^c,E^c\cap\Omega)
\end{split}
\end{equation*}
be the \emph{non-local $(K,\m)$-perimeter of $E$ relative to $\Omega$}.
Arguing exactly as in the proof of~\cite{DNRV15}*{Lem.~2.4}, if $\Omega_1,\Omega_2\in\mathscr A$ are such that $\m(\Omega_1\cap\Omega_2)=0$, then
\begin{equation*}
\begin{split}
P_{K,\m}&(E;\Omega_1\cup\Omega_2)
-
P_{K,\m}(E;\Omega_1)
-
P_{K,\m}(E;\Omega_2)
\\
&=
-
L_{K,\m}(E^c\cap\Omega_1,E\cap\Omega_1^c\cap\Omega_2)
-
L_{K,\m}(E^c\cap\Omega_2,E\cap\Omega_1\cap\Omega_2^c)
\end{split}
\end{equation*}
for any $E\in\mathscr A$.
In particular, if $K(x,y)>0$ for $(\m\otimes\m)$-a.e.\ $(x,y)\in X\times X$, $\m(E\cap\Omega_1^c\cap\Omega_2)>0$, and $\m(E\cap\Omega_1\cap\Omega_2^c)>0$, then
\begin{equation*}
P_{K,\m}(E;\Omega_1\cup\Omega_2)
<
P_{K,\m}(E;\Omega_1)
+
P_{K,\m}(E;\Omega_2)\,,
\end{equation*}
and thus, in particular, the map $\Omega\mapsto P_{K,\m}(E;\Omega)$ is not finitely additive.
The reader can easily check that this, in fact, occurs for the non-local perimeters~\eqref{eq:nonlocal_per} (assuming that $K>0$) and~\eqref{eq:frac_gauss_per}.
Hence, \cref{def:per_var_measures}, and consequently the subsequent construction of the Sobolev spaces (in particular, see \cref{subsec:sobolev_1_space}), cannot be applied to such non-local perimeter functionals.  
\end{remark}

\subsubsection{Distributional fractional perimeters}
In~\cite{CS19}, a new space $BV^s(\R^n)$ of functions with 
bounded fractional variation on $\R^n$ of order $s\in (0,1)$ is introduced via a distributional approach exploiting  suitable notions of fractional gradient and divergence. 
More precisely, the fractional $s$-variation of a function $u\in L^1(\R^n)$ is defined as
\[
|D^s u|(\R^n)=\sup\set*{\int_{\R^n}u\div^s\varphi\, \de x : \varphi\in \mathrm{C}^\infty_c(\R^n; \R^n),\  \left\|\varphi\right\|_{\infty}\leq 1},
\]
where
\[
\div^s\varphi(x)=\mu_{n,s}\int_{\R^n}\frac{(y-x)\cdot (\varphi(y)-\varphi(x))}{|y-x|^{n+s+1}}\,\de y,\quad x\in\R^n,
\]
and $\mu_{n,s}$ is a renormalization constant. 
The distributional fractional $s$-pe\-ri\-me\-ter of a Lebesgue measurable set $E\subset\R^n$ is then defined as the total fractional variation of its characteristic function $|D^s\chi_E|(\R^n)$.
In~\cite{CS19}, the authors show that
\[
|D^s\chi_E|(\R^n)\le \mu_{n,s} P_{s}(E)
\]
whenever $E$ is a measurable set, where  $P_s$ is as in~\eqref{eq:fractional_per}, thus showing that the distributional approach allows to extend the usual notion of fractional perimeter and enlarge the class of sets with finite fractional perimeter. 

Following \cite{CS19}, the functional $E\mapsto|D^s\chi_E|(\R^n)$ enjoys several properties on measurable sets of $\R^n$.
In fact, \ref{prop:empty} and~\ref{prop:space} are direct consequences of the definition, yielding the validity of the first part of \cref{res:relation_with_first_1_eigen}. Moreover, properties~\ref{prop:lsc}, \ref{prop:compact}, and~\ref{prop:isoperim} are, respectively, proved in \cite{CS19}*{Prop.~4.3}, \cite{CS19}*{Thm.~3.16}, and \cite{CS19}*{Thm.~4.4} (provided that $n\ge2$, see~\cite{CS19-2}*{Thm.~3.8} for the case $n=1$).
In addition, property~\ref{prop:complement_set} trivially follows from the definition. In particular, we can apply \cref{res:existence_N-Cheeger} and \cref{prop:min_PMC}.
Finally, the validity of \ref{prop:sub-mod} is open, whereas it is known that~\ref{propR:local} is false in general, see~\cite{CS19}*{Rem.~4.9}; thus, we cannot apply the results concerning the first eigenvalue of the general Dirichlet $p$-Laplacian.
It is also worth noticing that the local form of the chain rule in this context is false~\cite{CS22}, so a direct adaptation of the proof of \cref{thm:CheegerIn} in this framework is not clear.

We remark that the aforementioned results have never been proved before in this specific non-local setting.

\subsubsection{Non-local perimeter of Minkowski type}
Following~\cite{CMP12}, given $r>0$, for any $u\in L^1_{\loc}(\R^n)$, we let
\begin{equation*}
    \mathcal E_r(u)=\frac1{2r}\int_{\R^n}\operatorname{osc}_{B_r(x)}(u)\,\de x,
\end{equation*}
where 
\begin{equation*}
    \operatorname{osc}_A(u)=\esssup_Au-\essinf_Au
\end{equation*}
denotes the \emph{essential oscillation} of $u$ on the measurable set $A\subset\R^n$.
The functional $E\mapsto\mathcal E_r(\chi_E)$ is the \emph{non-local perimeter of Minkowski type} of the measurable set $E\subset\R^n$.
As discussed in~\cite{CMP12}, such perimeter functional meets properties~\ref{prop:empty}, \ref{prop:space}, 
\ref{prop:sub-mod},
\ref{prop:lsc}, 
and
\ref{prop:complement_set}, and naturally satisfies a coarea formula~\cite{CMP12}*{Eq.~(2.3)}.
Property~\ref{prop:isoperim} can be easily deduced from the isoperimetric inequality proved in~\cite{CDNV18}*{Lem.~1.12(i)}.
Finally, as observed in~\cite{CDNV18}*{Rem.~1.5}, property~\ref{prop:compact} does not hold.
As a consequence, our results allow to infer several properties of $N$-Cheeger sets for $N\ge1$ (if they exist) in the perimeter-measure space $(\R^n, \mathscr B(\R^n),\mathscr L^n, \mathcal E_r)$.

\subsection{Riemannian manifolds}
Let $(M,g)$ be a complete Riemannian manifold of dimension $n\in\N$. 
When endowed with its distance, it is a separable metric space and its volume measure is a non-negative Borel measure that is finite on bounded Borel sets, so that one can rely on the discussion made in \cref{subsec:mms} to obtain the validity of \ref{prop:empty}, \ref{prop:space}, \ref{prop:sub-mod}, \ref{prop:lsc}, \ref{prop:compact}, and~\ref{prop:complement_set}.
Concerning~\ref{prop:isoperim}, the basic result contained in \cref{res:mira_isop} can be refined in several ways.  If $M$ has non-negative Ricci curvature, then property~\ref{prop:isoperim} is a consequence of the sharp isoperimetric inequality recently obtained in~\cite{Brendle22} for non-compact manifolds with Euclidean volume growth, also see~\cite{AFM20}. If $M$ is compact, then property~\ref{prop:isoperim} can be deduced from the celebrated  L\'evy--Gromov isoperimetric inequality,  see~\cite{Gromov}*{App.~C}.
If the Ricci curvature bound is negative, no global isoperimetric inequality can be derived without further assumptions on the manifold, such as lower bounds on the diameter of $M$, see~\cite{Milman15} and the references therein for a more detailed discussion.

Following the strategy presented in~\cref{subsec:mms}, the results contained in our paper then allow to recover Cheeger inequalities in Riemannian manifolds with non-negative curvature, in the spirit of the original appearance of Cheeger inequalities in compact Riemannian manifolds, due to Cheeger \cite{Cheeger70} for $p=2$. 
The results of our paper also cover the existence of Cheeger sets, originally proved in~\cite{Buser82} for compact Riemannian manifolds (see also~\cite{Benson16}), and the links with the the prescribed mean curvature.
We refer to~\cite{CGL15} for the relation between the Cheeger constant and the torsion problem~\eqref{eq:torsional_creep_PDE} for $p=2$ in compact Riemannian manifolds.

\subsection{\texorpdfstring{$\mathsf{CD}$}{CD}-spaces}
$\mathsf{CD}$-spaces are metric-measure spaces generalizing Riemannian manifolds with Ricci curvature bounded from below, via assumptions on a synthetic notion of curvature, encoded in the so-called \emph{curvature-dimension condition} $\mathsf{CD}(K,n)$ for $K\in\R$ and $n\ge 1$, see the cornerstones~\cites{SturmI06,SturmII06,LV09}.
Geometric Analysis on these non-smooth spaces is subject to a great interest in the recent years, see~\cites{AGS14,CM17,CM21,Ambrosio18}. 

As $\mathsf{CD}(K,n)$ spaces for $K\in\R$ and $n\ge 1$ are complete metric spaces endowed with a Borel measure $\m$ that is finite on bounded Borel sets, 
the discussion made in \cref{subsec:mms} applies to this framework, yielding the validity of properties \ref{prop:empty}, \ref{prop:space}, \ref{prop:sub-mod}, \ref{prop:lsc}, \ref{prop:compact}, and~\ref{prop:complement_set}.
Concerning~\ref{prop:isoperim}, the simple inequality provided by \cref{res:mira_isop} can be refined in several ways.  For  $K\geq 0$, a sharp isoperimetric inequality has been recently proved in~\cite{APPS22} for the subclass of $\mathsf{RCD}(0,n)$-spaces with Euclidean volume growth, when $\m=\mathcal H^n$, yielding \ref{prop:isoperim} with 
$f(\varepsilon)=n\omega_n^{1/n}\AVR(X)^{1/n}\varepsilon^{-1/n}$. Here, $\AVR(X)$ stands for the asymptotic volume ratio, assumed to be in $(0,1]$. 
We also refer to~\cite{APP22}*{Thm.~3.19 and Rem.~3.20},  where the validity of property~\ref{prop:isoperim} is discussed in more general metric-measure spaces with particular attention to the case of $\mathsf{CD}(K,n)$ spaces for $K\in\mathbb R$ and $1<n<+\infty$, and to the celebrated L\'evy--Gromov isoperimetric inequality proved in~\cites{CM17,CM18} holding for essentially non-branching $\mathsf{CD}(K,n)$ spaces with finite diameter. 

The equivalence of the Cheeger constant and the first $1$-eigenvalue of the Laplacian was previously pointed out in~\cite{CM17b}*{Sect.~5} for more general metric-measure spaces including non-branching $\mathsf{CD}(K,n)$ spaces. Lower bounds on the Cheeger constant for (essentially non-branching) $\mathsf{CD}^*(K,n)$ spaces are proved in \cites{CM17,CM17b,CM18}. 

\subsection{Carnot--Carath\'eodory spaces}

Let $\omega\subset\R^n$ be a non-empty connected open set and let $\mathcal X=\{X_1,\dots,X_k\}$ be vector fields in $\omega$ with real $\mathrm{C}^\infty$-smooth coefficients. An absolutely continuous curve $\gamma\colon[0,T]\to\omega$ is \emph{admissible} if there exists $u=(u_1,\dots,u_k)\in L^1([0,T])$ such that 
\[
\dot\gamma(t)=\sum_{i=1}^ku_i(t)X_i(\gamma(t)).
\]
Given two points $x,y\in\omega$, we let $d_{cc}(x,y)$ be the \emph{Carnot--Carath\'eodory distance} between $x$ and $y$, defined as the shortest length of admissible curves connecting them. 
We assume that the \emph{H\"ormander condition} 
\[\operatorname{rank}(\operatorname{Lie}\mathcal X)=n\] 
on the Lie algebra $\operatorname{Lie}\mathcal X$ generated by $\mathcal X$ holds true. Then, $d_{cc}(x,y)<+\infty$ for any couple of points $x,y\in\omega$ thanks to Chow--Rashewski Theorem, see~\cite{ABB} for the details.
The metric space $(\omega,d_{cc})$ is called a \emph{Carnot--Carath\'eodory space}, and it is separable. 
Assuming $(\omega,d_{cc})$ to be also complete and endowing it with the Lebesgue measure $\mathscr L^n$, one is then allowed to rely on the discussion of~\cref{subsec:mms} to guarantee the validity of properties \ref{prop:empty}, \ref{prop:space}, \ref{prop:sub-mod}, \ref{prop:lsc}, \ref{prop:compact}, and~\ref{prop:complement_set} for the distributional perimeter of \cref{def:BV_in_metric_spaces}.
One can see that this actually corresponds to the so-called \emph{$\mathcal X$-perimeter}, introduced in~\cite{CDG94} and then systematically studied in~\cites{FSSC96,GN96}. 

We discuss the validity of~\ref{prop:isoperim}. We first observe that, as summarized in~\cite{HK00}*{Sect.~11.4}, up to taking a smaller $\omega$,
we are ensured (globally in $\omega$) the validity of a doubling property for metric balls and of a $(1,1)$-Poincar\'e inequality for the \emph{horizontal gradient} $\nabla_{\mathcal X} u=\sum_{i=1}^kX_iuX_i$, thanks to the celebrated works~\cites{NSW85,Jerison86}. This allows to rely on the results of \cite{GN96}*{Thm.~1.18}, guaranteeing the validity of 
the following isoperimetric inequality for any Lebesgue measurable set $E\subset \omega$
\[
P_\mathcal X(E)\geq C_{\mathcal X} \mathscr L^n(E)^{\frac{Q-1}{Q}}.
\]
Here $C_\mathcal {X}$ is a positive constant depending on $\omega$ and $\mathcal X$, and $Q\geq n$ is the so-called \emph{homogeneous dimension}. Property \ref{prop:isoperim} then follows with $f(\varepsilon)=C_\mathcal{X}\varepsilon^{-1/Q}$.
This refines the basic inequality given by \cref{res:mira_isop}.

All the results contained in our paper then apply to this setting, establishing existence of Cheeger sets (\cref{res:existence_N-Cheeger}), relations with the the prescribed curvature functional (\cref{prop:min_PMC}), and Cheeger inequalities following the strategy presented in~\cref{subsec:mms}.
We observe that, following \cite{FHK99}*{Cor.~11}, and recalling that the topology induced by $d_{cc}$ is equivalent to the Euclidean one~\cite{ABB}*{Thm.~3.31}, the Sobolev space $\RW^{1,p}(\omega,\mathscr L^n)$ introduced in \cref{def:weak_p_slope} is given by
\[
\RW^{1,p}(\omega,\mathscr L^n)=\set*{u\in L^p(\omega,\mathscr L^n): |\nabla_{\mathcal X} u|\in L^p(\omega,\mathscr L^n)},
\]
where $|\nabla_{\mathcal X} u|=\sqrt{\sum_{i=1}^k(X_iu)^2}$.
In particular, a sub-Riemannian version of Green's identity ensures that, for an admissible bounded open set $\Omega$ and $u\in W^{1,2}(\Omega)$ with Dirichlet boundary conditions on $\Omega$, we have
\[
\int_\Omega |\nabla_{\mathcal X}u |^2\,\de x\, \de y=-\int_\Omega u \Delta_{\mathcal X} u \,\de x\, \de y,
\]
where $\Delta_{\mathcal X} u =\sum_{i=1}^k X_i^*X_iu$ is the so-called hypoelliptic sub-Laplacian associated with $\mathcal X$ and $X^*_i$ the formal adjoint of $X_i$.
In particular, \eqref{eq:CheegerIn} gives a lower bound for the bottom of the spectrum of $-\Delta_{\mathcal X}$ on $\Omega$. Cheeger's inequalities of this type have already been investigated in~\cite{Montefalcone13} in the context of Carnot groups. This paper extends them to more general Carnot--Carath\'eodory structures.

\subsection{Metric graphs}

Let $G=(V,E)$ be a connected compact graph (for simplicity, with no loops or multiple
edges).

We identify each edge $e\in E$ with an ordered pair $(i_e,f_e)$, denoting the initial and the final vertices of $e$, and we assume the existence of an increasing bijection $c_e\colon e\to[0,\ell_e]$, for some length $\ell_e\in(0,+\infty]$, such that $c_e(i_e)=0$ and $c_e(f_e)=\ell_e$, and we let $x_e=c_e(x)$ be the coordinate of the point $x\in e$.
In this case, $G$ is said a \emph{metric graph}.

A function on $G$ is identified with a collection of functions defined on $(0,\ell_e)$ for each $e\in E$, so that
\begin{equation*}
\int_G u(x)\,\de x
=
\sum_{e\in E}
\int_0^{\ell_e}[u]_e(x)\,\de x,
\end{equation*}
where $[u]_e$ is the function $u$ defined on the edge $e\in E$.
Following~\cite{Mazon}*{Def.~2.9}, the total variation of $u\in BV(G)$ is defined as 
\begin{equation}
\label{eq:mazon_var}
\Var_G(u)
=
\sup\set*{\int_G u(x)\,z'(x)\,\de x : z\in\mathcal K(G),\ \|z\|_{\infty}\le 1},
\end{equation}
where 
\begin{equation*}
\mathcal K(G)
=
\set*{z : \sum_{e\in E}\|[z]_e\|_{W^{1,2}(0,\ell_e)}<+\infty,\ 
\sum_{e\in E_v}[z]_e(v)\,\nu^e(v)=0,\, \forall v\in \mathrm{int}\, V},
\end{equation*}
being $E_v$ the set of edges incident to $e$, $\nu^e(i_e)=-1$ and $\nu^e(f_e)=1$, and $\mathrm{int}\, V$ the set of vertices with more than one incident edge.
Accordingly, the perimeter of $E\subset G$ is given by $P_G(E)=\Var_G(\chi_E)$.

Properties~\ref{prop:empty} and~\ref{prop:space} directly follow from the definition in~\eqref{eq:mazon_var}, whereas \ref{prop:lsc}, \ref{prop:compact}, \ref{prop:isoperim}, and \ref{prop:complement_set} are established in \cite{Mazon}*{Prop.~2.12}, \cite{Mazon}*{Thm.~2.6}, and~\cite{Mazon}*{Rem.~2.10}, respectively. 
Finally, \ref{prop:sub-mod} can be proved arguing as in \cite{AFP00book}*{Prop.~3.38} using the coarea formula given by~\cite{Mazon}*{Thm.~2.13}. 
As a consequence, \cref{res:existence_N-Cheeger} ensures existence of $N$-Cheeger sets for any $N\in\N$, thus generalizing~\cite{Mazon}*{Thm.~3.2} to the case of $N$-Cheeger sets.

\vspace{2ex}
\textbf{Conflict of interest.}
On behalf of all authors, the corresponding author states that there is no conflict of interest.


\bibliography{biblio} 

\end{document}